\documentclass{article}
\usepackage{amssymb, amsmath, mathrsfs, amsthm}
\usepackage{graphicx}
\usepackage{color}
\usepackage{authblk}
\usepackage[top=1in, bottom=1in, left=1in, right=1in]{geometry}
\usepackage{float, caption, subcaption}
\newtheorem{theorem}{Theorem}[section]
\newtheorem{lemma}[theorem]{Lemma}

\newtheorem{conjecture}{Conjecture}[section]

\begin{document}

\title{Gallai's Conjecture for Complete and ``Nearly Complete'' Graphs}

\author[1]{Hua Wang}
\author[2]{Andrew Zhang}

\affil[1]{Mathematical Sciences, Georgia Southern University }
\affil[2]{Wayzata High School}
\date{August 2022}
\maketitle

\begin{abstract}
The famous Gallai's Conjecture states that any connected graph with $n$ vertices has a path decomposition containing at most $\left\lfloor \frac{n+1}{2} \right\rfloor$ paths.
In this note, we explore graphs generated from removing edges from complete graphs. We first provide an explicit construction for a path decomposition of complete graphs that satisfies Gallai’s Conjecture. We then use that construction to prove that we can remove stars and certain tadpoles such that the resulting graph still satisfies Gallai’s Conjecture. We also introduce a potential general approach through analyzing non-isomorphic path decompositions of complete graphs.
\end{abstract}

\section{Introduction}\label{sec:definition}

One of the primary areas of study in graph theory is the decomposition of graphs.  We say that a set of subgraphs $S = \{ S_1, S_2, \dots S_n \}$ is a \textit{decomposition} of the original graph $G$ if every edge in $G$ is in exactly one of the subgraphs in $S$. Researchers have studied decomposing graphs into double stars \cite{El15}, trees \cite{Gal59}, and collections of paths and cycles \cite{Aru13}. 


In this paper, we consider the decomposition of graphs into paths, which we refer to as {\it path decompositions}. The study of path decompositions has been primarily focused on Gallai's Conjecture as stated below.

\begin{conjecture}[\cite{Gal59}] Any connected graph with $n$ vertices has a path decomposition containing at most $\lfloor \frac{n+1}{2} \rfloor$ paths .
\end{conjecture}

Initially stated in 1959 \cite{Gal59}, Gallai's Conjecture is still open to this day. The first fundamental piece of progress on Gallai's Conjecture was Lovasz's proof that any graph $G$ with $n$ vertices could be decomposed into at most $\lfloor \frac{n+1}{2} \rfloor$ paths and cycles \cite{lov68}.  Since this result by Lovasz in 1968, further work has been done in two main directions: analysis using E-subgraphs and bounds on the maximal degree of a graph.

The E-subgraph of a graph $G$, denoted as ${G}_e$, is the subgraph formed when all vertices of odd degree are removed from the graph. The first significant advancement using the E-subgraph was made by Pyber \cite{lp96}, who proved that if $G_e$ was a forest, then $G$ satisfied Gallai's Conjecture. 
This line of work was extended by Fan in 2005 \cite{GFan05}, who proved that if each block of $G_e$ was a triangle free graph of maximum degree at most 3, then Gallai's Conjecture held for $G$. Furthermore, Botler and Sambinelli \cite{Botler20} generalized Fan's result by proving that as long as each block in $G_e$ had maximum degree at most $3$ or $G_e$ had at most one block that contained triangles, $G$ satisfied Gallai's Conjecture. 

Another direction involves bounds on the maximum degree of graphs. For example, it was shown in 2019 that all graphs with maximum degree of at most 5 satisfy Gallai's Conjecture \cite{bon19}. Similarly, it was shown that all graphs of maximum degree 6, bar a few exceptions, satisfy Gallai's Conjecture \cite{chu21}. 

All of the above mentioned work seem to heavily restrict the degree of the vertices in the graphs. It is natural to consider Gallai's Conjecture in graphs without such constraints. 
In this paper, we consider path decompositions of complete and ``nearly complete'' graphs of any size, hence providing a new avenue of analyzing Gallai's Conjecture that does not rely on bounds on the degree of graphs.

In Section~\ref{sec:results}, we present explicit decompositions to show that all complete graphs satisfy Gallai's Conjecture. Then, in Section~\ref{sec:expand}, we show that stars and tadpoles can be removed from complete graphs so that Gallai's Conjecture still holds. This allows us to find a larger family of graphs that satisfy Gallai's Conjecture, but are still not bounded by a maximal degree.  Finally, in Section~\ref{sec:con}, we briefly summarize our findings and discuss some observations arised from trying to prove Gallai's Conjecture based on removing the path ends from a given path decomposition. 

\section{Complete graphs}\label{sec:results}
This section discusses Gallai's Conjecture for complete graphs. We will divide complete graphs into three categories: 1) those with $2k$ vertices, 2) those with $2k+1$ vertices where $k$ is even, and 3) those with $2k+1$ vertices where $k$ is odd. We will prove Gallai's Conjecture by presenting explicit constructions that will be used in later sections.
\subsection{Gallai's Conjecture for $K_n$ with even $n$}
\begin{lemma} \label{Gallaievencomplete}
Gallai's Conjecture holds for complete graphs with an even number of
vertices.
\end{lemma}
\begin{proof}
We use Walecki Decomposition \cite{alspach08} in this proof. Without loss of generality, we assume the vertices of our even complete graph are placed at the vertices of a regular $n$-gon.

Walecki Decomposition consists of first forming a zigzag path going through every vertex on the even complete graph and then rotating that zigzag to create new paths until the entire complete graph is covered. Figure \ref{fig:zigzag} gives an example for a complete graph on $6$ vertices. We first create a zigzag as shown in \ref{fig:singlezigzag}, and then rotate that zigzag twice. This gives the path decomposition of the complete graph as shown in \ref{fig:fulldecomp}.

\begin{figure}[htp]
    \begin{subfigure}[t]{0.4\textwidth}
         \centering
                  \includegraphics[height = 5cm]{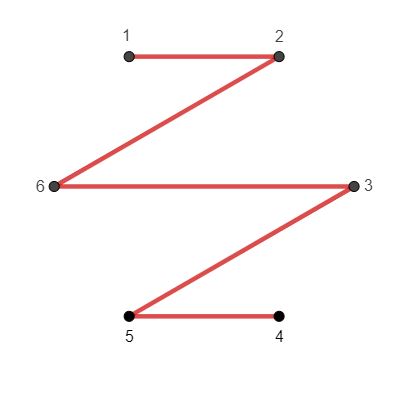}
         \caption{A zigzag on 6 vertices}
         \label{fig:singlezigzag}
     \end{subfigure}
\hskip 0.1in
     \begin{subfigure}[t]{0.48\textwidth}
         \centering
         \includegraphics[height=6cm]{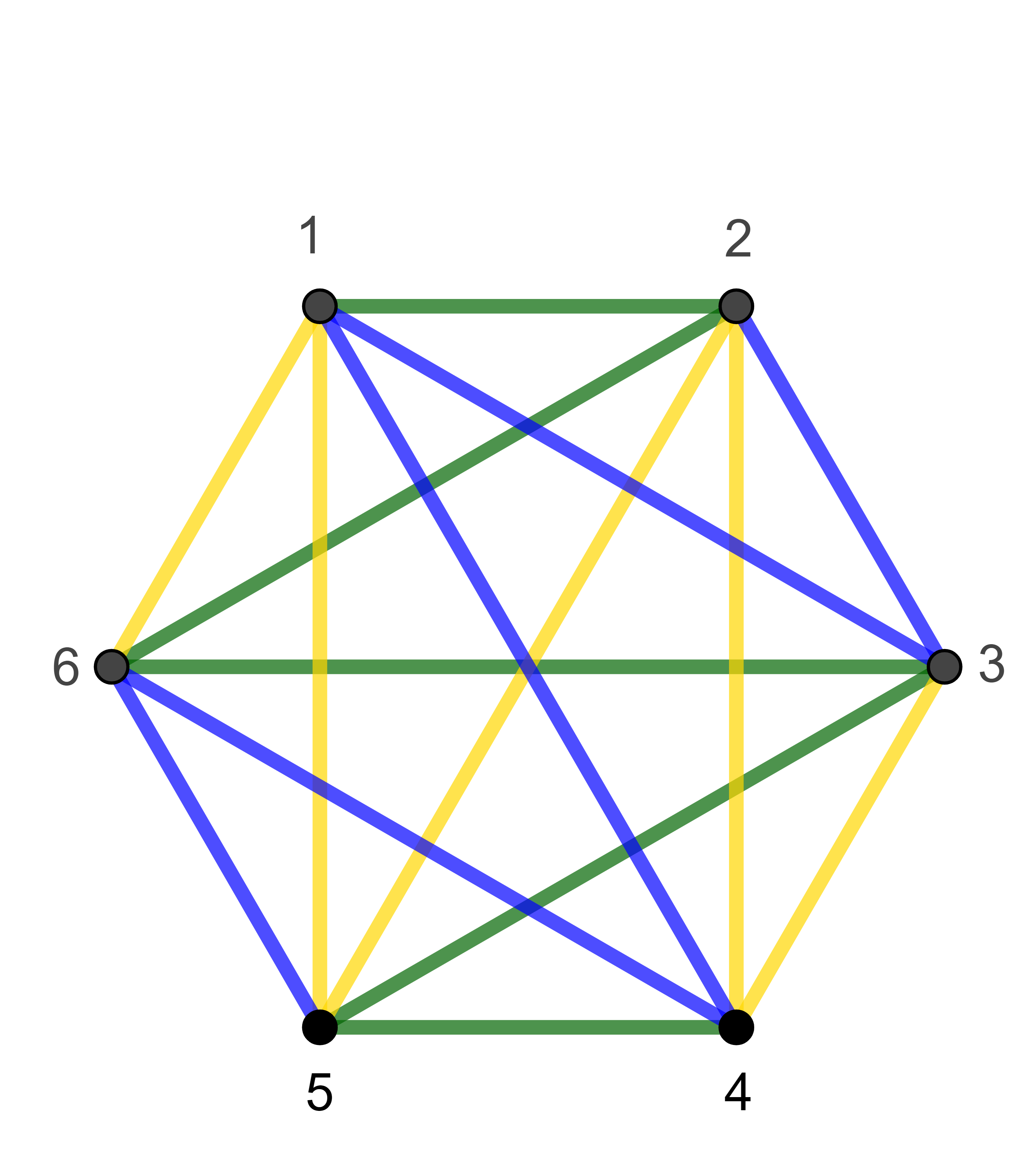}
         \caption{A Walecki Decomposition for $K_6$}
         \label{fig:fulldecomp}
     \end{subfigure}
    \caption{Path decomposition for $K_6$}
    \label{fig:zigzag}
\end{figure}

For even $n$, this strategy uses exactly  $\frac{n}{2} = \lfloor \frac{n+1}{2} \rfloor$ paths, satisfying Gallai's Conjecture.
\end{proof}


\subsection{Gallai's Conjecture for $K_n$ with odd $n$}
Letting $n=2k+1$, we will subdivide this into cases where $k$ is even and $k$ is odd. 

\subsubsection{Odd $k$}

\begin{lemma}\label{Thm: Complete graph decomp for odd k}
Gallai's Conjecture holds for an odd complete graph with $2k+1$ vertices where $k$ is  odd.
\end{lemma}
\begin{proof}
We  start by considering a complete graph on $2k$ vertices labeled sequentially from $1$ to $2k$. Note that we can build the path decomposition in Lemma \ref{Gallaievencomplete}. 

We will transform this path decomposition into a path decomposition of a complete graph with $2k+1$ vertices by adding $2k$ additional edges (from vertex $2k+1$ to all other vertices) to the existing paths.

First, define a rerouting of the edge $v_iv_j$ to a vertex $v_m$ where $v_iv_m$ and $v_jv_m$ do not exist as removing $v_iv_j$ and then adding the edges $v_iv_m$ and $v_jv_m$. Since the edges $v_1v_2, v_3v_4,$ $ \dots, v_{2k-1}v_{2k}$ all belong to different paths in the path decomposition, when we reroute them to $v_{2k+1}$, no path will visit $v_{2k+1}$ twice. 

Rerouting the edges from vertices $v_1v_2, v_3v_4,$ $ \dots, v_{2k-1} v_{2k}$ all to $v_{2k+1}$ creates the $2k$ connections needed between $v_{2k+1}$ and the other vertices. Now, we remove the edges $v_2v_3, v_4v_5, \dots ,$\\$v_{2k-2}v_{2k-1}$ and add the path $v_1v_2\dots v_{2k}$.

Thus, we have constructed a path decomposition for complete graphs with $2k+1$ vertices (where $k$ is odd) with $\frac{n+1}{2}$ paths, confirming Gallai's Conjecture.
\end{proof}

Figure~\ref{fig: 7 vertex complete graph} illustrates the process described above for $n=7$ and $k=3$.

\begin{figure}[htp]
    \centering
     \begin{subfigure}[t]{0.3\textwidth}
         \centering
         \includegraphics[width=\textwidth]{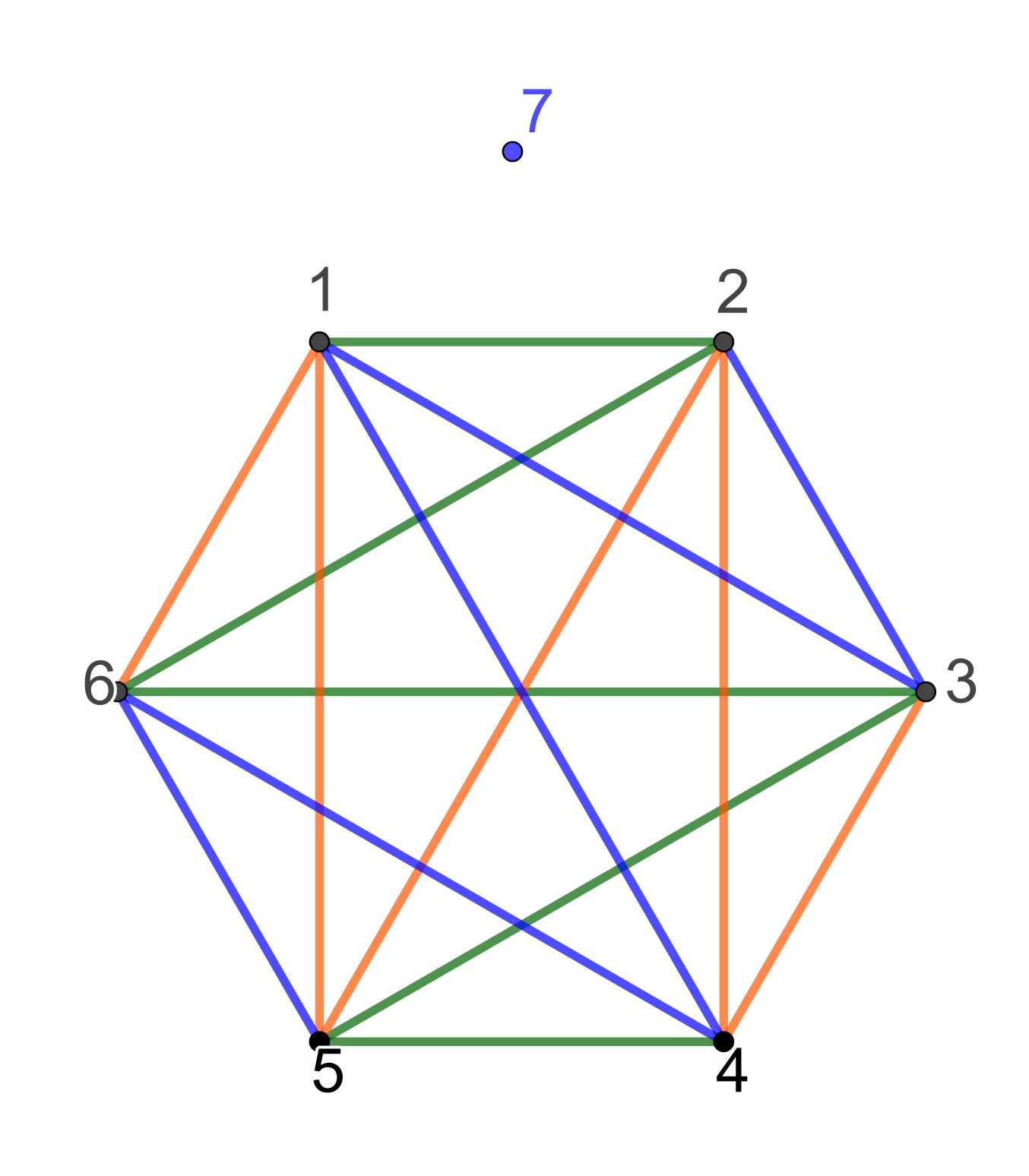}
         \caption{Adding one vertex onto $K_6$}
         \label{Fig: 7a}
     \end{subfigure}
     \hskip 0.2in
     \begin{subfigure}[t]{0.3\textwidth}
         \centering
         \includegraphics[width=\textwidth]{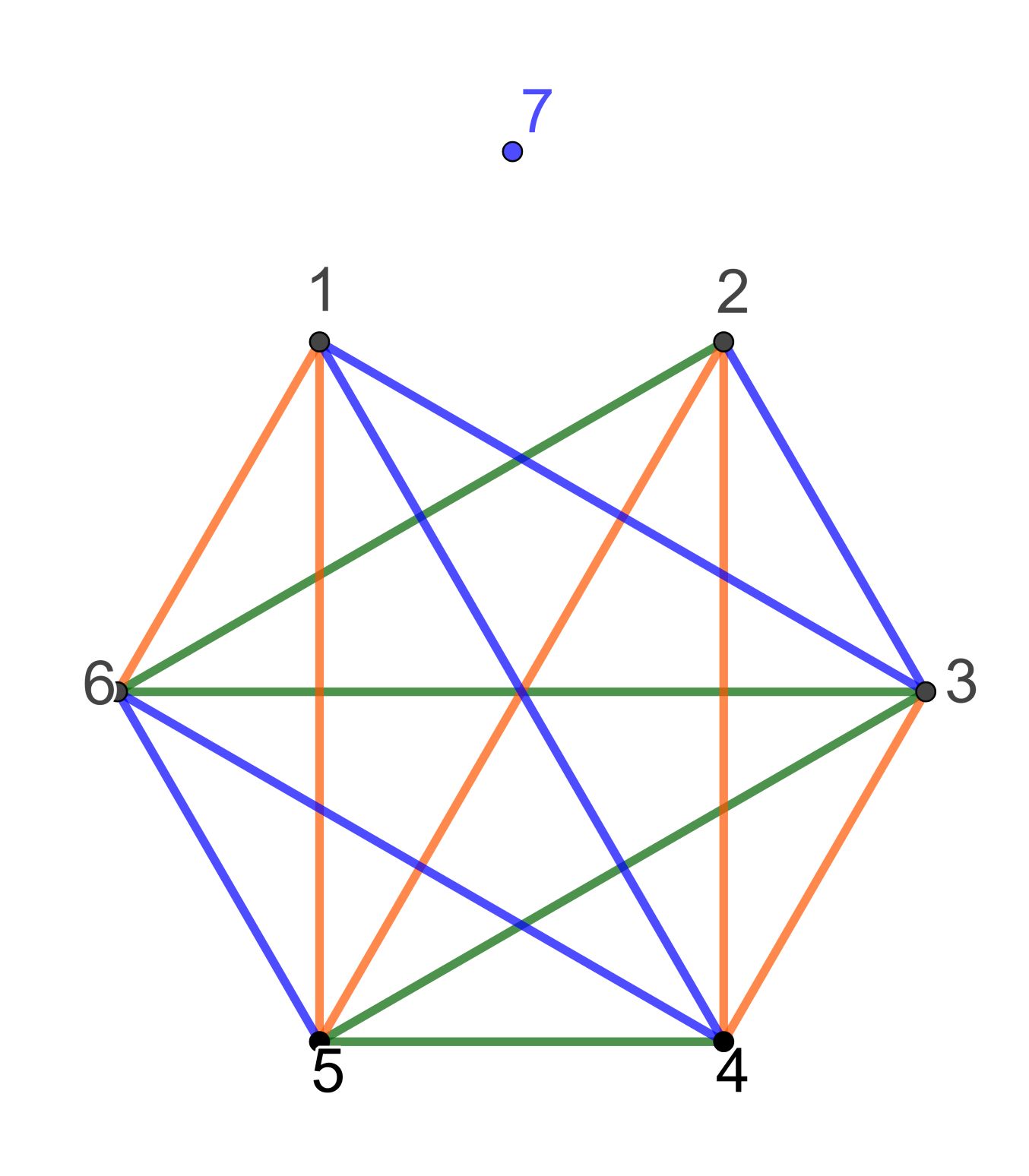}
         \caption{Removing $v_1v_2$}
         \label{Fig: 7b}
     \end{subfigure}
\hskip 0.2in
     \begin{subfigure}[t]{0.3\textwidth}
         \centering
         \includegraphics[width=\textwidth]{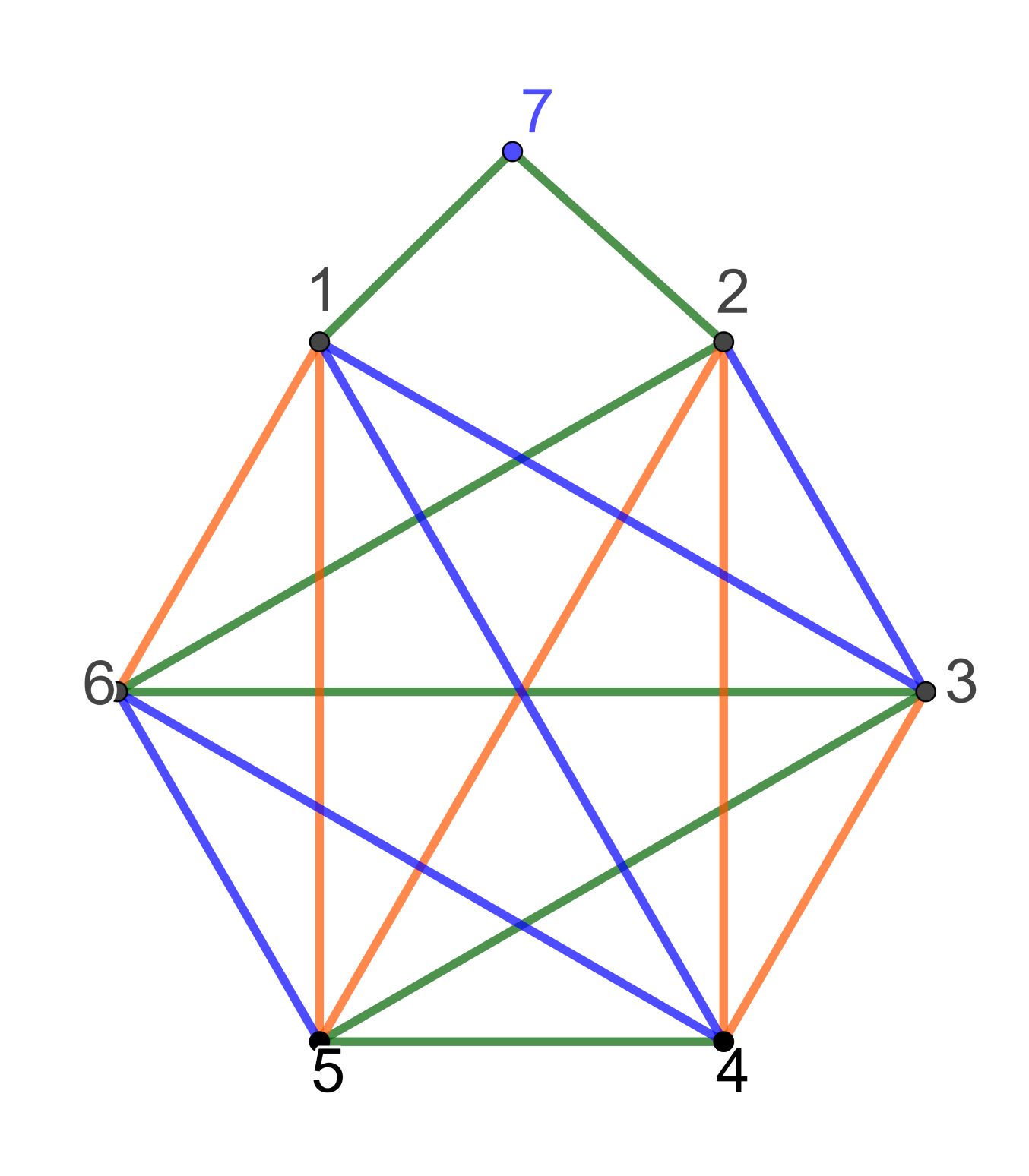}
         \caption{Adding in $v_1v_7$ and $v_2v_7$}
         \label{Fig: 7c}
     \end{subfigure}

        \begin{subfigure}[t]{0.3\textwidth}
         \centering
         \includegraphics[width=\textwidth]{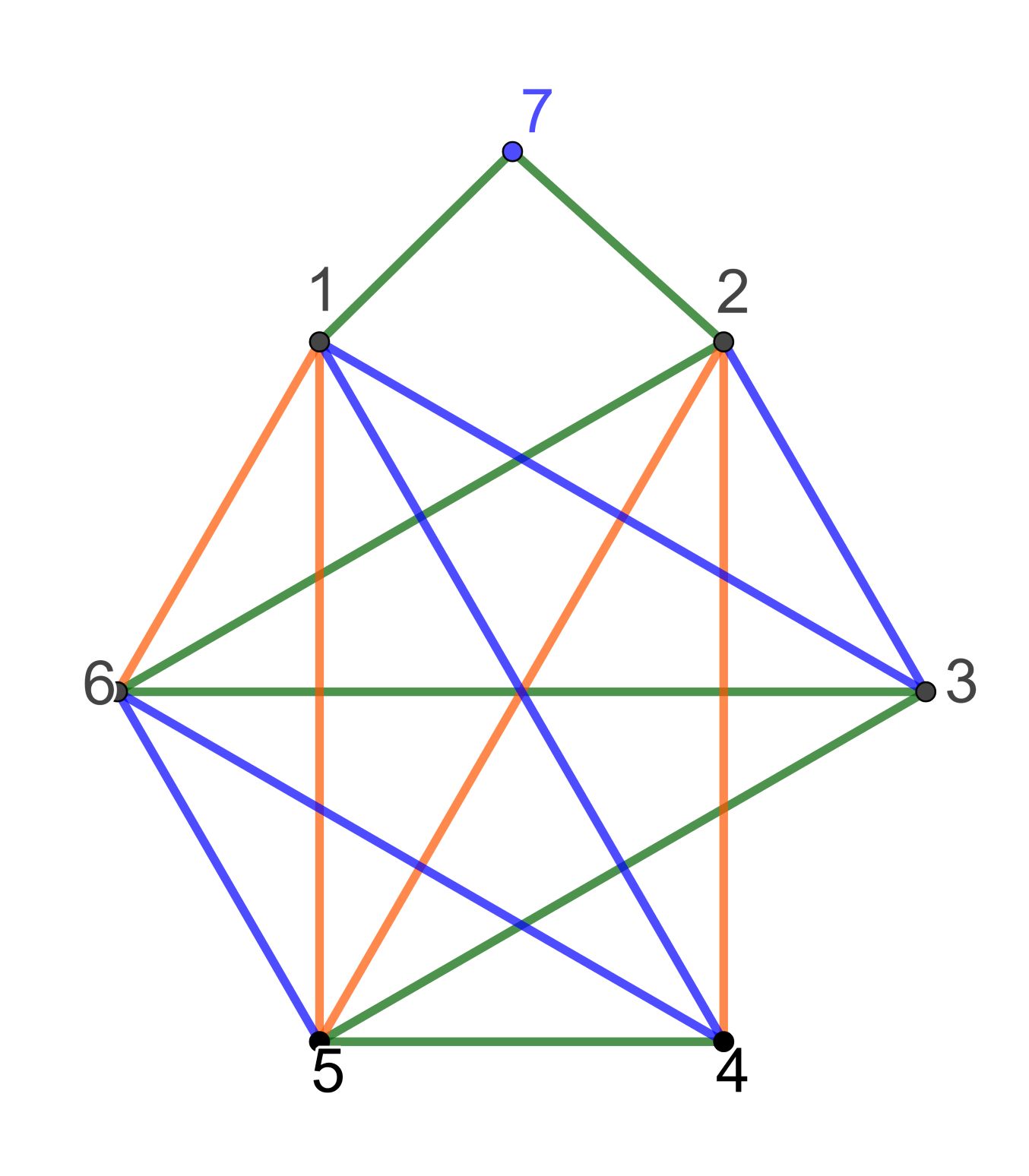}
         \caption{Removing $v_3v_4$}
         \label{Fig: 7d}
     \end{subfigure}
     \hskip 0.2in
     \begin{subfigure}[t]{0.3\textwidth}
         \centering
         \includegraphics[width=\textwidth]{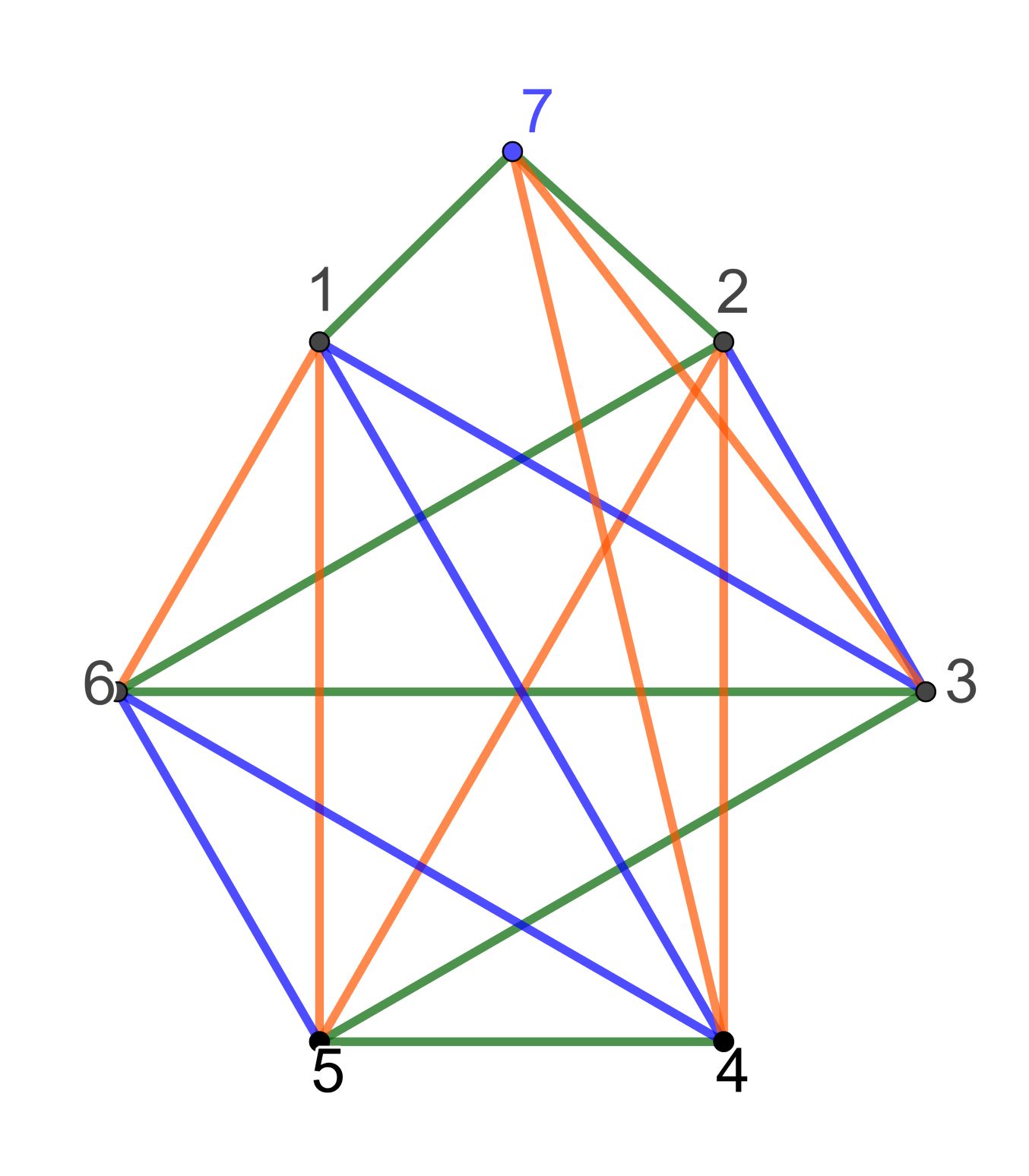}
         \caption{Adding in $v_3v_7$ and $v_4v_7$}
         \label{Fig: 7e}
     \end{subfigure}
     \hskip 0.2in
     \begin{subfigure}[t]{0.3\textwidth}
         \centering
         \includegraphics[width=\textwidth]{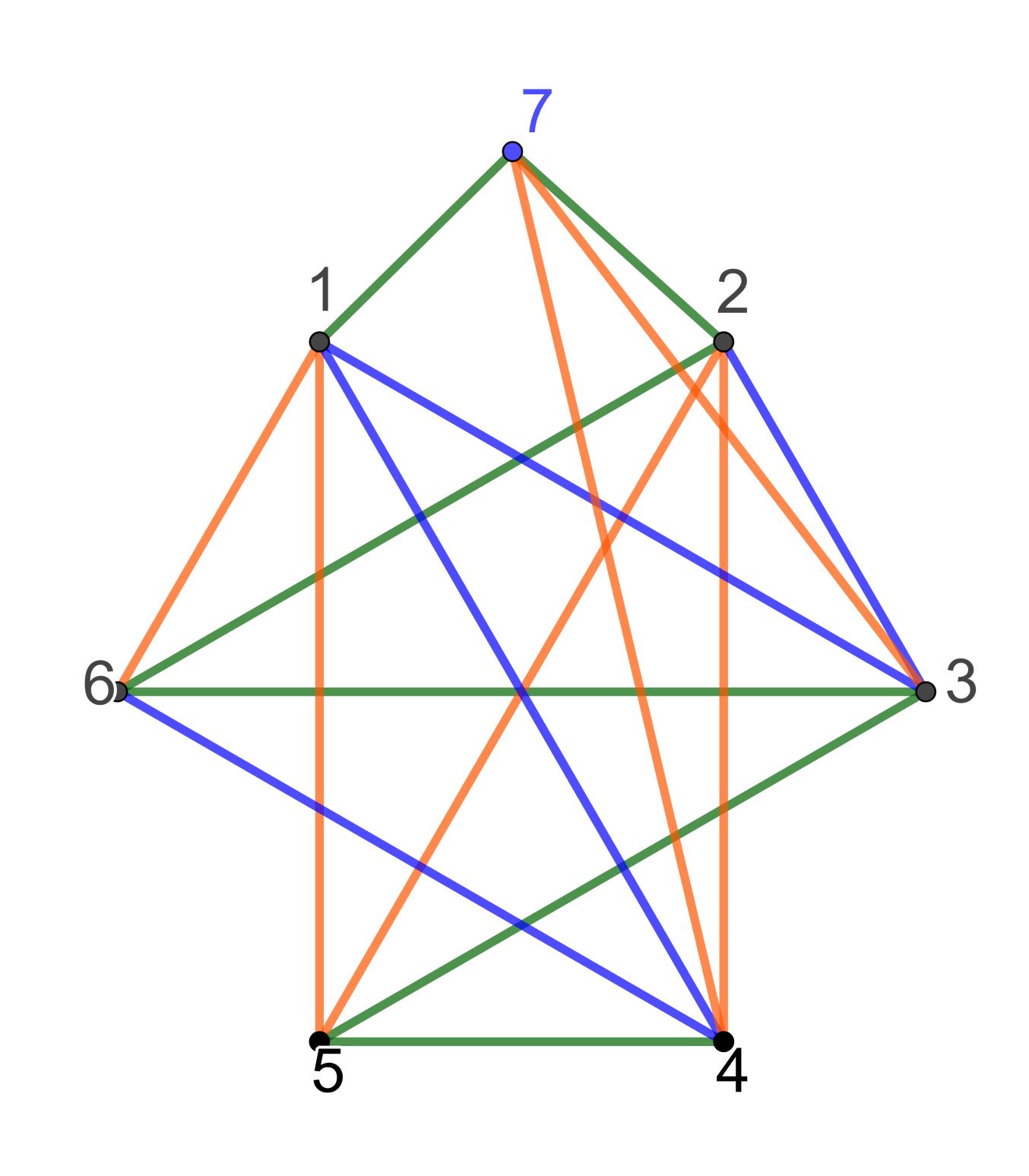}
          \caption{Removing $v_5v_6$}
          \label{Fig: 7f}
     \end{subfigure}

          \begin{subfigure}[t]{0.3\textwidth}
         \centering
         \includegraphics[width=\textwidth]{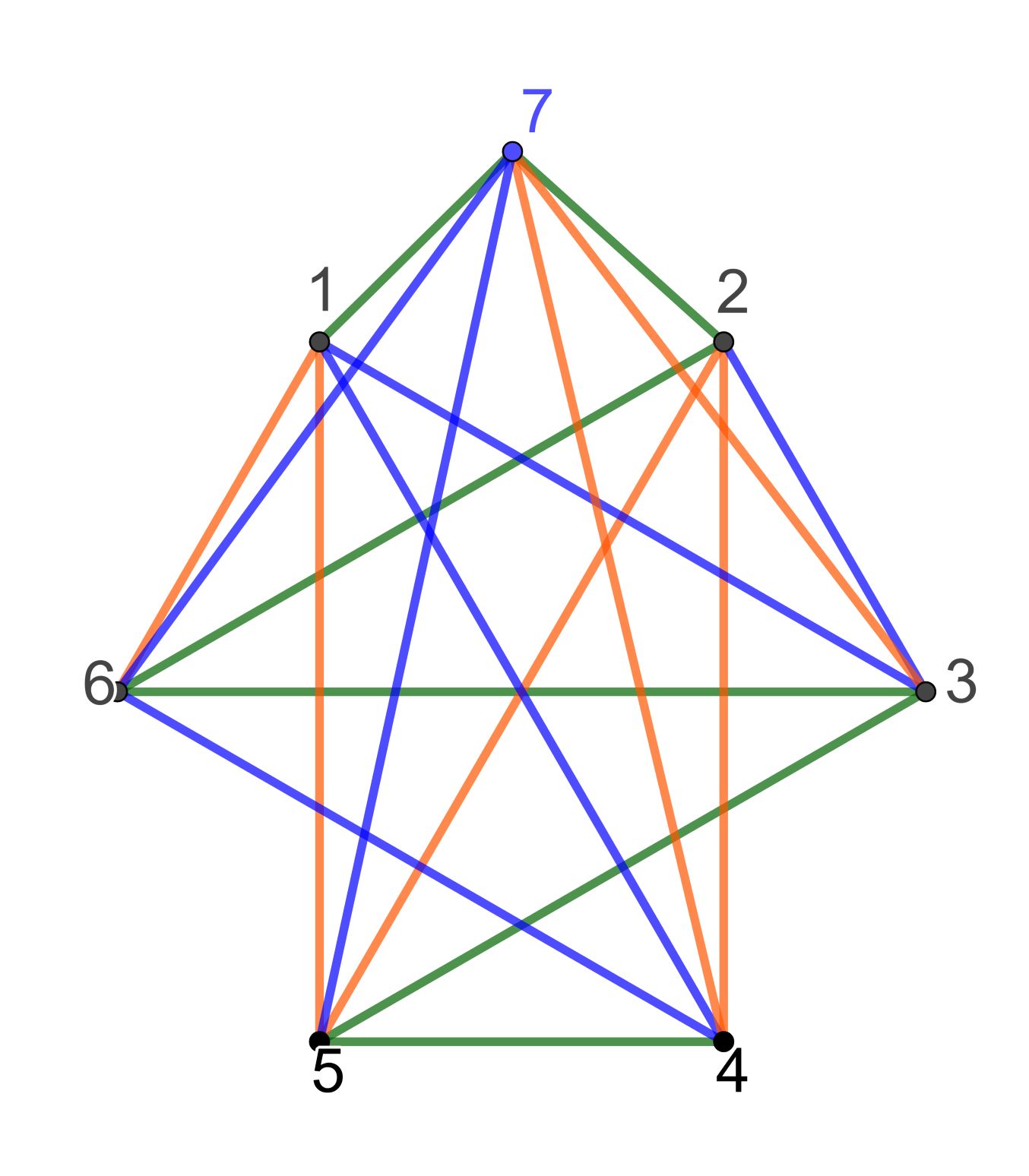}
          \caption{Adding in $v_5v_7$ and $v_6v_7$ }
          \label{Fig: 7g}
     \end{subfigure}
       \hskip 0.3in
          \begin{subfigure}[t]{0.3\textwidth}
         \centering
         \includegraphics[width=\textwidth]{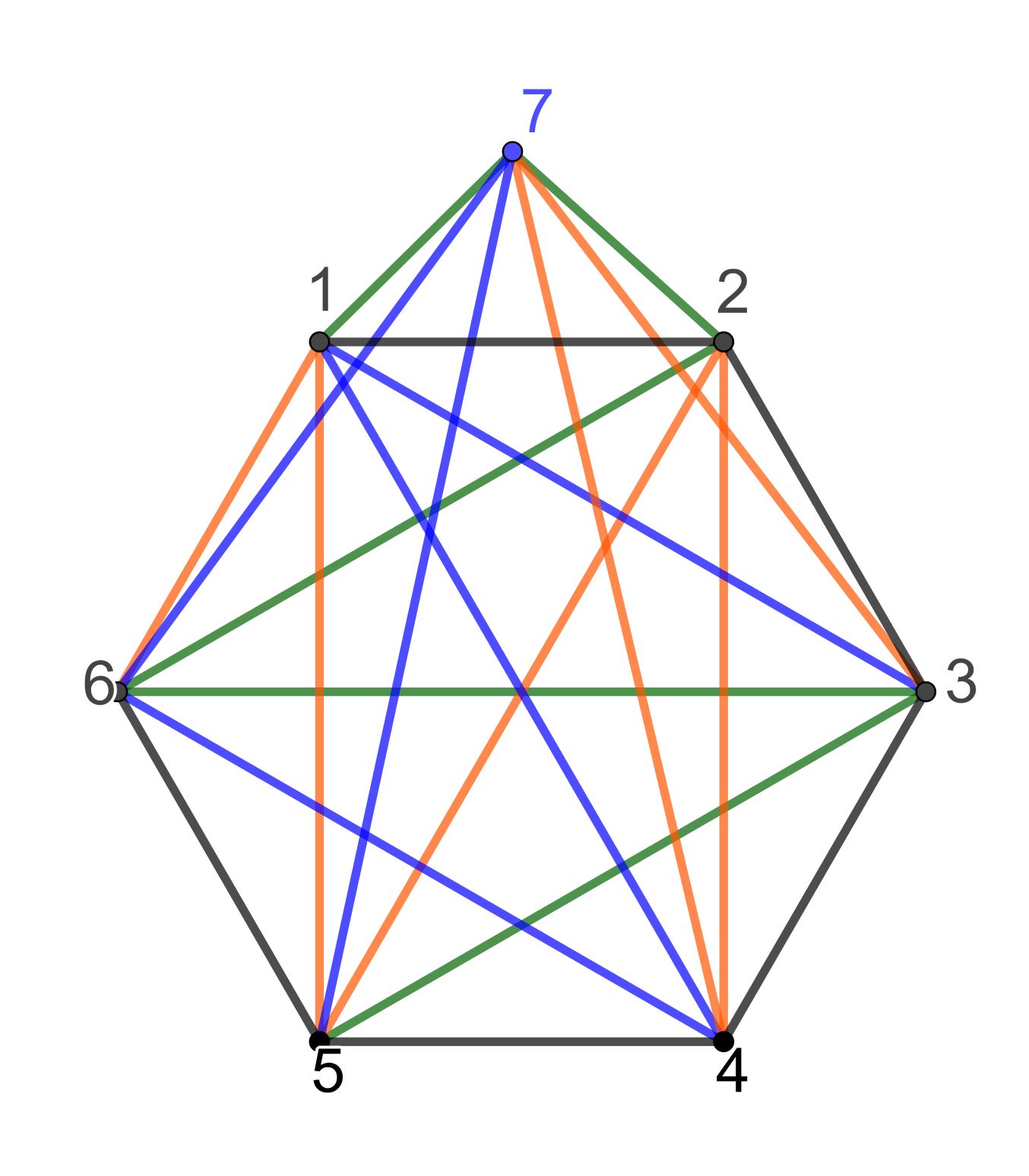}
         \caption{Removing $v_2v_3$ and $v_4v_5$, Adding $v_1v_2v_3v_4v_5v_6$}
         \label{Fig: 7h}
     \end{subfigure}
    \caption{Path decomposition of $K_{7}$}
    \label{fig: 7 vertex complete graph}
\end{figure}

\subsubsection{Even $k$}

\begin{lemma}\label{Thm: Complete graph decomp for even k}
Gallai's Conjecture holds for an odd complete graph with $2k+1$ vertices where $k$ is even.
\end{lemma}

\begin{proof}

We still will use the path decomposition generated in Lemma \ref{Gallaievencomplete} for the even complete graph on $2k$ vertices. 

In this case, we first reroute $\frac{k}{2}$ edges to $v_{2k+1}$ by rerouting $v_1v_2, v_3v_4, \dots, v_{k-1}v_k$. Then, we reroute $\frac{k}{2}-1$ edges by rerouting $v_{k+2}v_{k+3}, v_{k+4}v_{k+5}, \dots, v_{2k-2}v_{2k-1}$. We are left with $v_{k+1}$ and $v_{2k}$ that do not have edges to  $v_{2k+1}$. To create these two edges, we create a path that starts at  $v_{2k}$, goes to  $v_{2k+1}$, and then goes back to  $v_{k+1}$. To form the rest of the path we add in the edges $v_{k+1}v_k, v_{k}v_{k-1}, \dots, v_2v_1$ and the edges $v_{2k}v_{2k-1}, v_{2k-1}v_{2k-2}, \dots, v_{k+3}v_{k+2}$. If these edges already belong to some paths, they are simply removed from those paths (note that they are the ``ends'' of those paths). This gives a path decomposition with $\frac{n+1}{2}$ paths, verifying Gallai's Conjecture.
\end{proof}

Figure~\ref{fig:9Point} shows an example of the constructed path decomposition for $9$ vertices. 
\begin{figure}[htp]
    \centering
    \includegraphics[height=8cm]{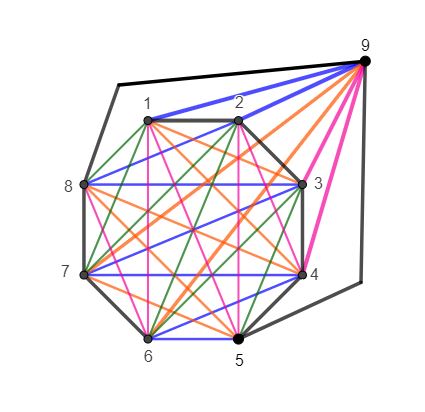}
    \caption{Path decomposition of $K_9$}
    \label{fig:9Point}
\end{figure}

\section{``Nearly complete'' graphs}\label{sec:expand}
In this section, we will use our previously constructed path decompositions to show that we can remove stars and certain tadpoles from complete graphs and still have a graph that satisfies Gallai's Conjecture. This contributes towards our goal of finding general results regarding graphs that satisfy Gallai's Conjecture without degree bounds.

First, note that because the vertices of a graph can be relabeled without loss of generality, we only need to prove that we can remove a subgraph from one part of the graph and we will know that we can remove that subgraph from anywhere on the graph.

\subsection{Removing stars} 

\begin{theorem}
Given $n\geq 2$, Gallai's Conjecture holds for any graph $G$ obtained from removing a star on $m\leq n-2$ edges from $K_n$. 
\end{theorem}

\begin{proof}
We will consider two different cases depending on the parity of $n$.
\begin{itemize}
 \item $n=2k+1$:

First, considering when $k$ is odd, we will start with the complete graph path decomposition created in Lemma \ref{Thm: Complete graph decomp for odd k}. An ``end edge'' is the edge on the end of a path in our decomposition, and a ``side edge'' is the edge adjacent to the end edge. Note that we can remove the end edge without disconnecting the rest of the path.

In our path decomposition for the complete graph, $v_{2k+1}$ joins the end edge and the side edge of $k$ paths. Thus, we can remove the $k$ end edges without disconnecting any paths. Additionally, after removing all $k$ end edges, we can remove up to $k-1$ of the adjacent edges. This means that we can remove a star with up to $2k-1$ edges centered at $v_{2k+1}$ without disconnecting any paths. The remaining graph and path decomposition verifies Gallai's Conjecture.

Figure \ref{fig:starremoval} shows the removal of a star on $5$ edges from  $K_7$.

When $k$ is even, we again start with the complete graph path decomposition created in Lemma \ref{Thm: Complete graph decomp for even k}. Note that $v_{2k+1}$ joins the end edge and the side edge of $k-1$ paths. To remove two more edges from $v_{2k+1}$, we remove $v_{2k}v_{2k+1}$ and $v_{k+1}v_{2k+1}$, and then reconnect that path by changing the end edge $v_{2k}v_{1}$ to be on that path. Now, note that these removals will suffice to remove stars with up to $2k-1$ edges centered at $v_{2k+1}$ without disconnecting any paths, again verifying Gallai's Conjecture for star removal in this case.
\begin{figure}[htp]
    \centering
     \begin{subfigure}[t]{0.32\textwidth}
         \centering
         \includegraphics[width=\textwidth]{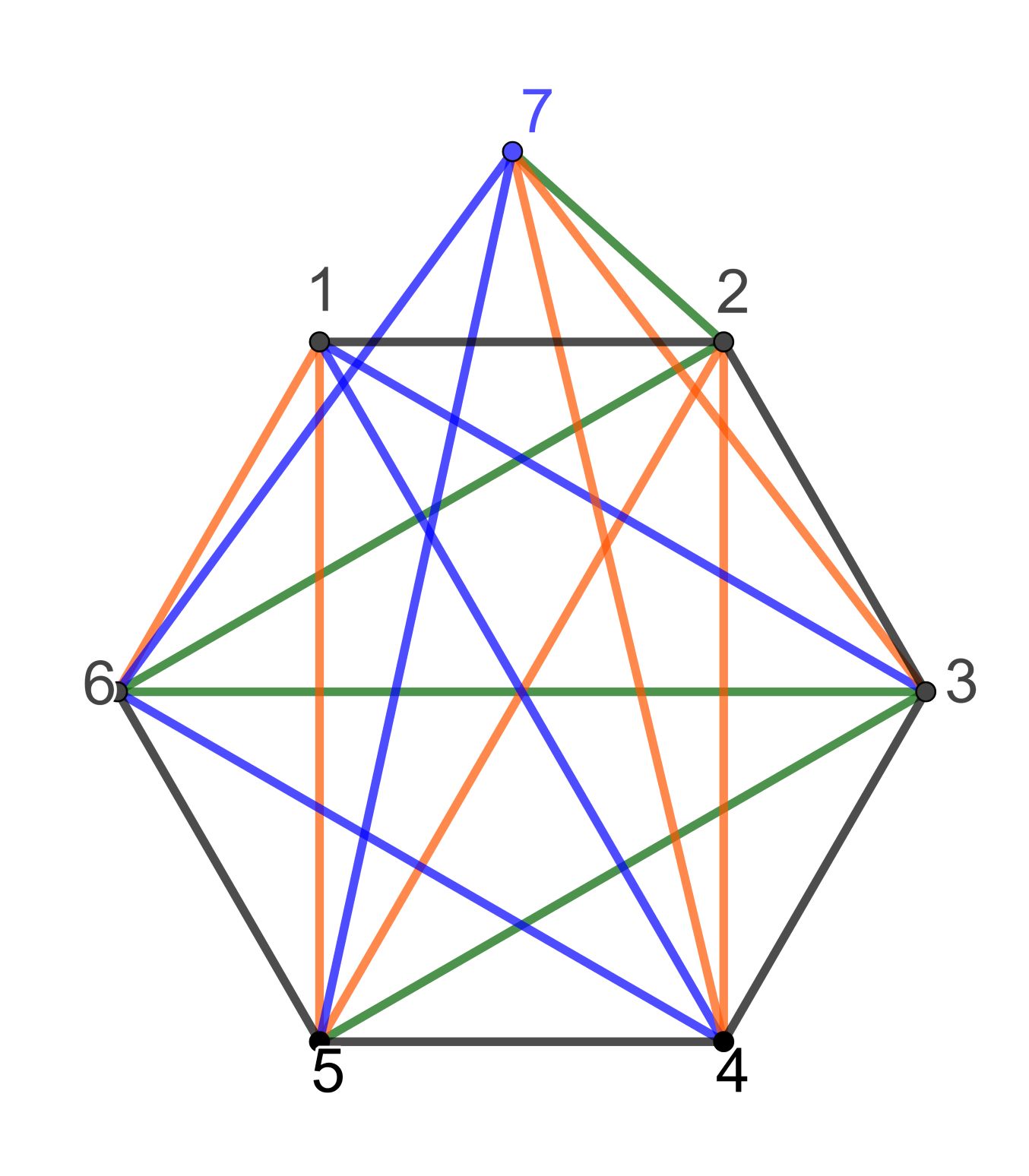}
         \caption{Removing the end edge of the green path connected to $v_7$}
     \end{subfigure}
     \hfill
     \begin{subfigure}[t]{0.32\textwidth}
         \centering
         \includegraphics[width=\textwidth]{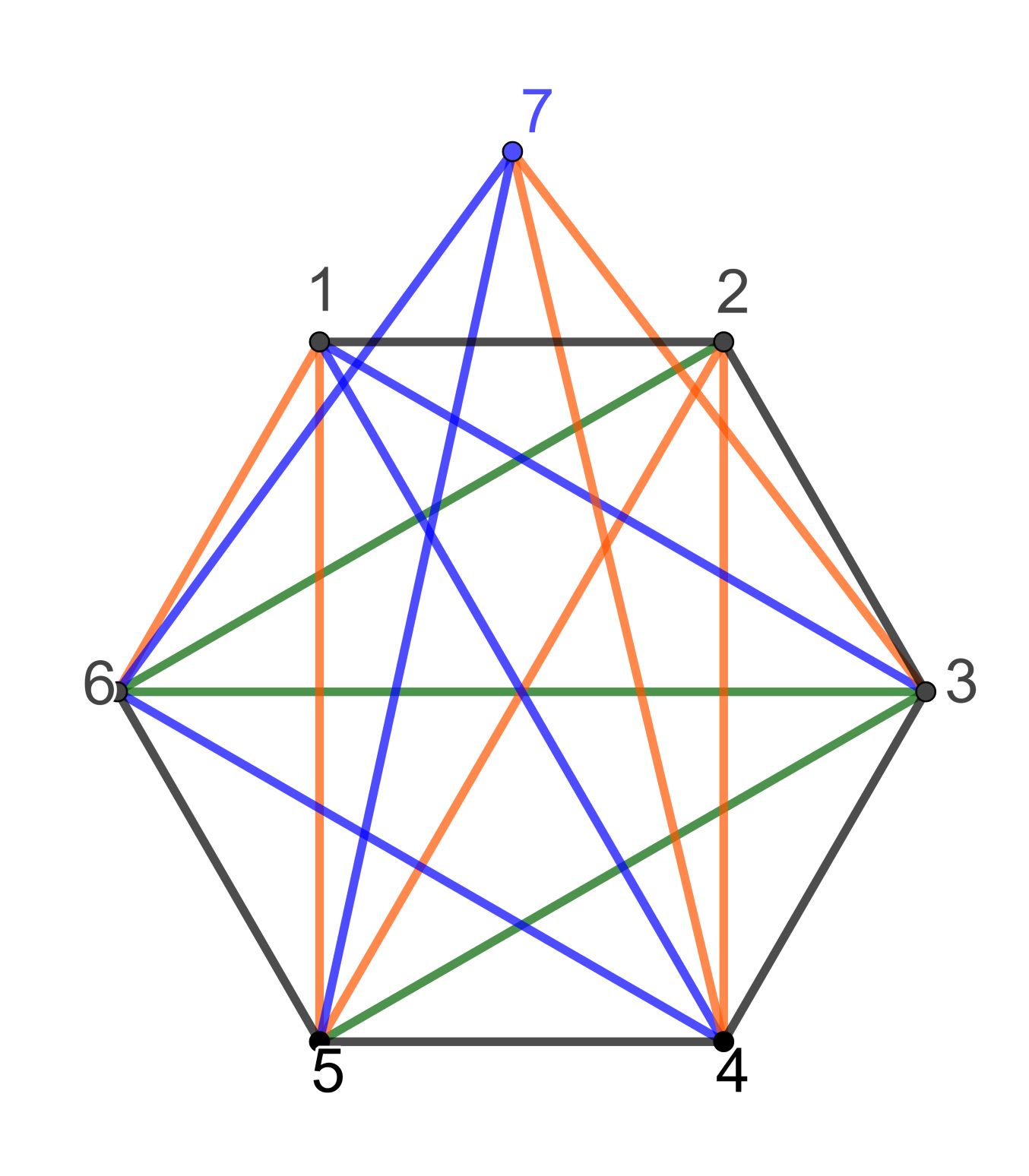}
         \caption{Remove the adjacent edge of the green path connected to $v_7$}
     \end{subfigure}
     \hfill
     \begin{subfigure}[t]{0.32\textwidth}
         \centering
         \includegraphics[width=\textwidth]{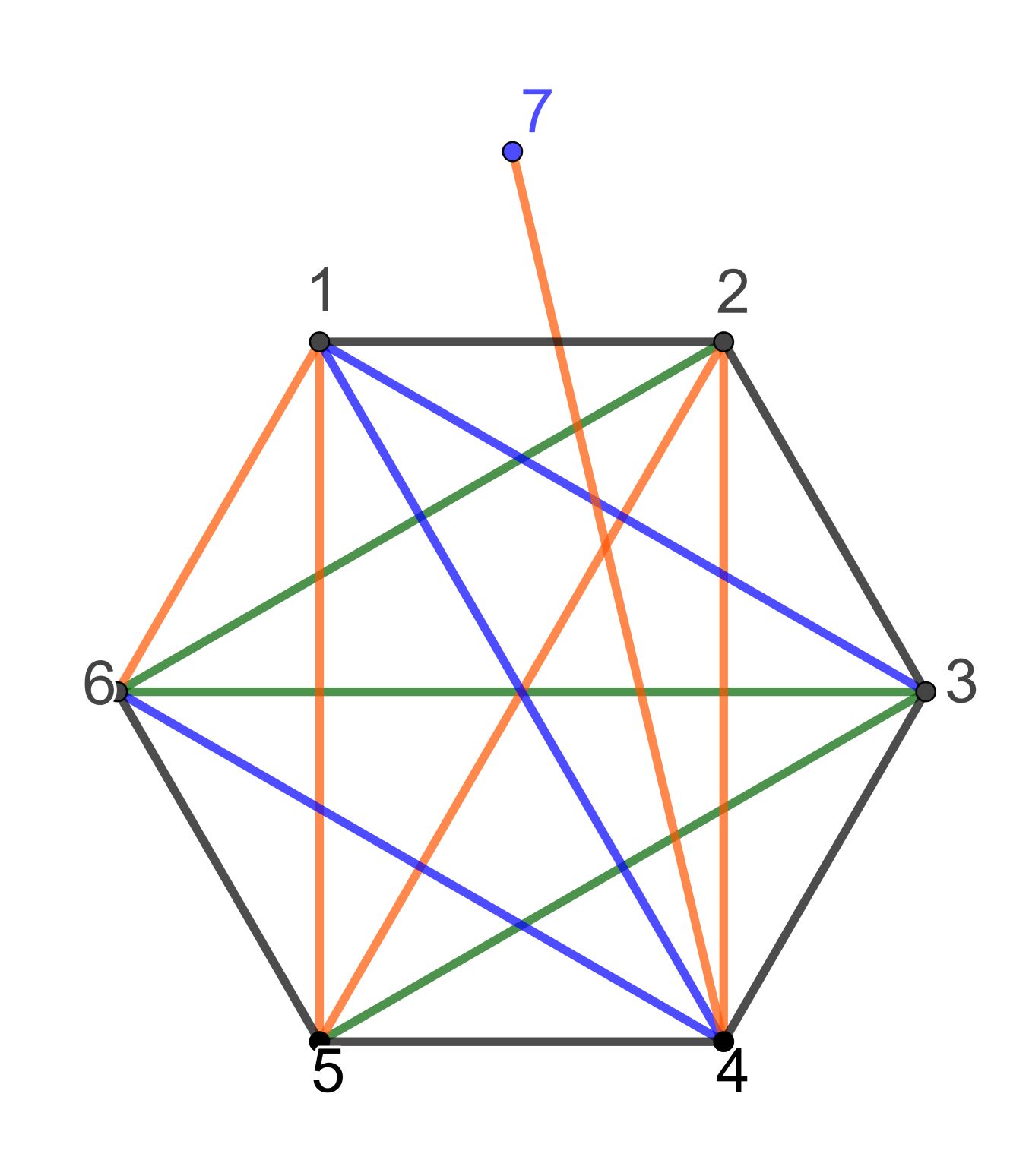}
         \caption{Removing the end and adjacent edge of the blue path and the end edge of orange path}
     \end{subfigure}
    
    \caption{Star removal example: Removing a star on $5$ edges from  $K_7$}
    \label{fig:starremoval}
\end{figure}

\item $n=2k$:

First, we call a pair of adjacent edges on a path ``fork'' with the ``center'' of the fork being the common vertex. We will also call the edge connecting the two ends of a fork the ``base''.

We will now consider the original path decomposition for $K_{2k}$. Consider an arbitrary vertex $v$. Note that $k-1$ paths have a fork centered at $v$. Additionally, note that the base of the fork is an end edge. Thus, we can remove the fork and change the path that the base is on to reconnect the path that the fork was removed from (Figure~\ref{fig:edgeremoval}).

\begin{figure}[htp]
    \centering
     \begin{subfigure}[t]{0.32\textwidth}
         \centering
         \includegraphics[width=\textwidth]{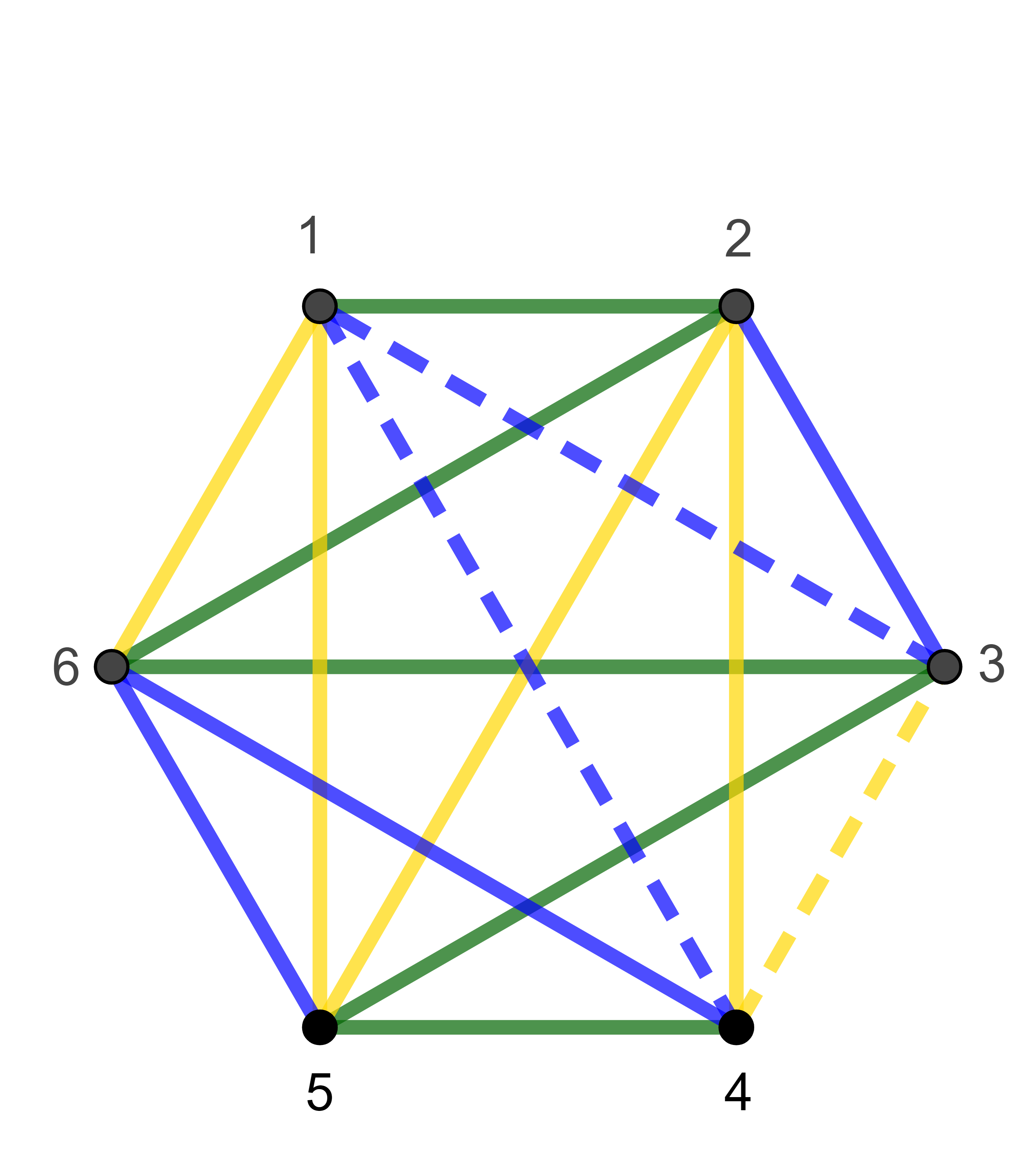}
         \caption{Dotted blue fork with yellow base}
     \end{subfigure}
     \hfill
     \begin{subfigure}[t]{0.32\textwidth}
         \centering
         \includegraphics[width=\textwidth]{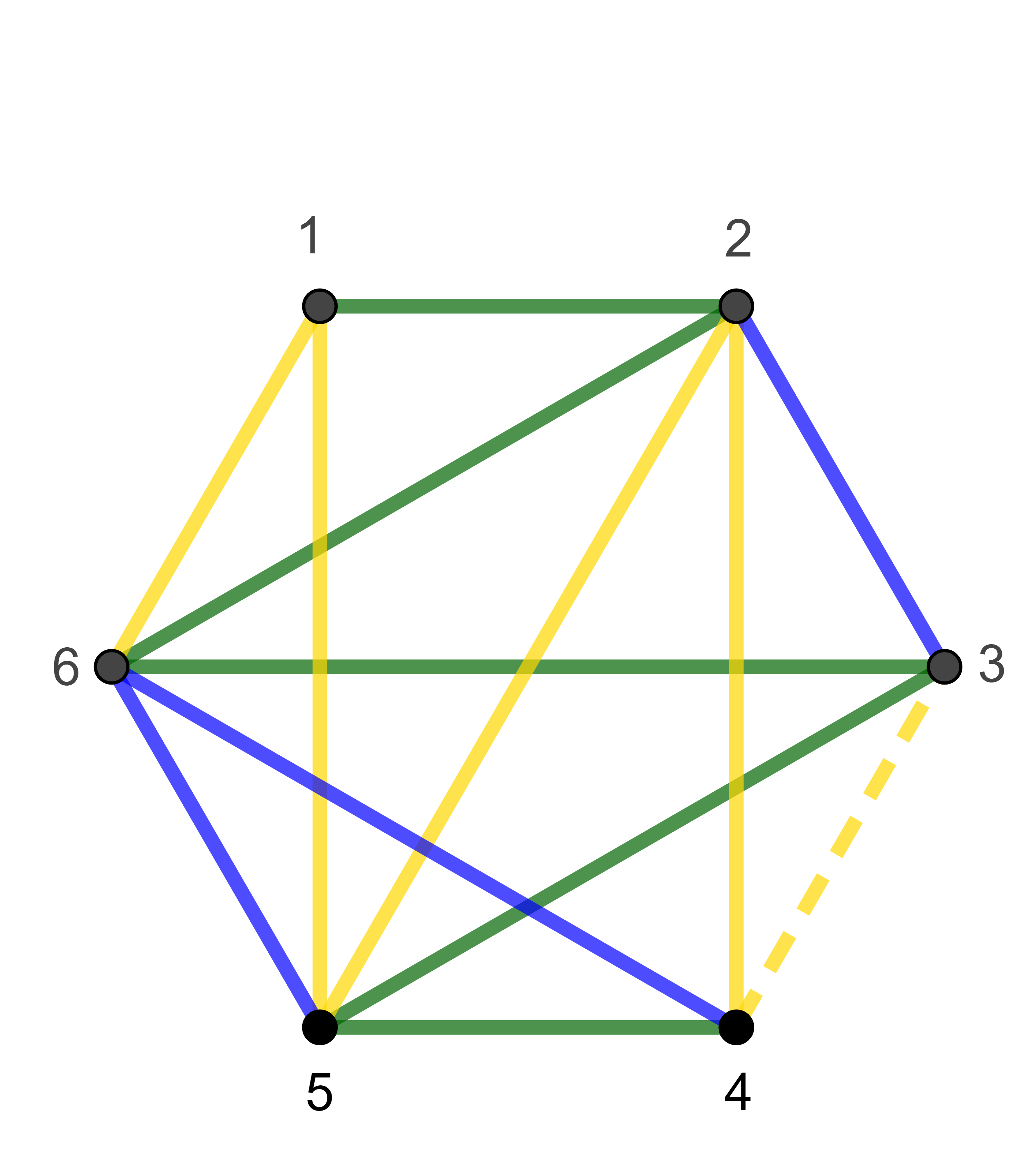}
         \caption{Removing blue fork}
     \end{subfigure}
     \hfill
     \begin{subfigure}[t]{0.32\textwidth}
         \centering
         \includegraphics[width=\textwidth]{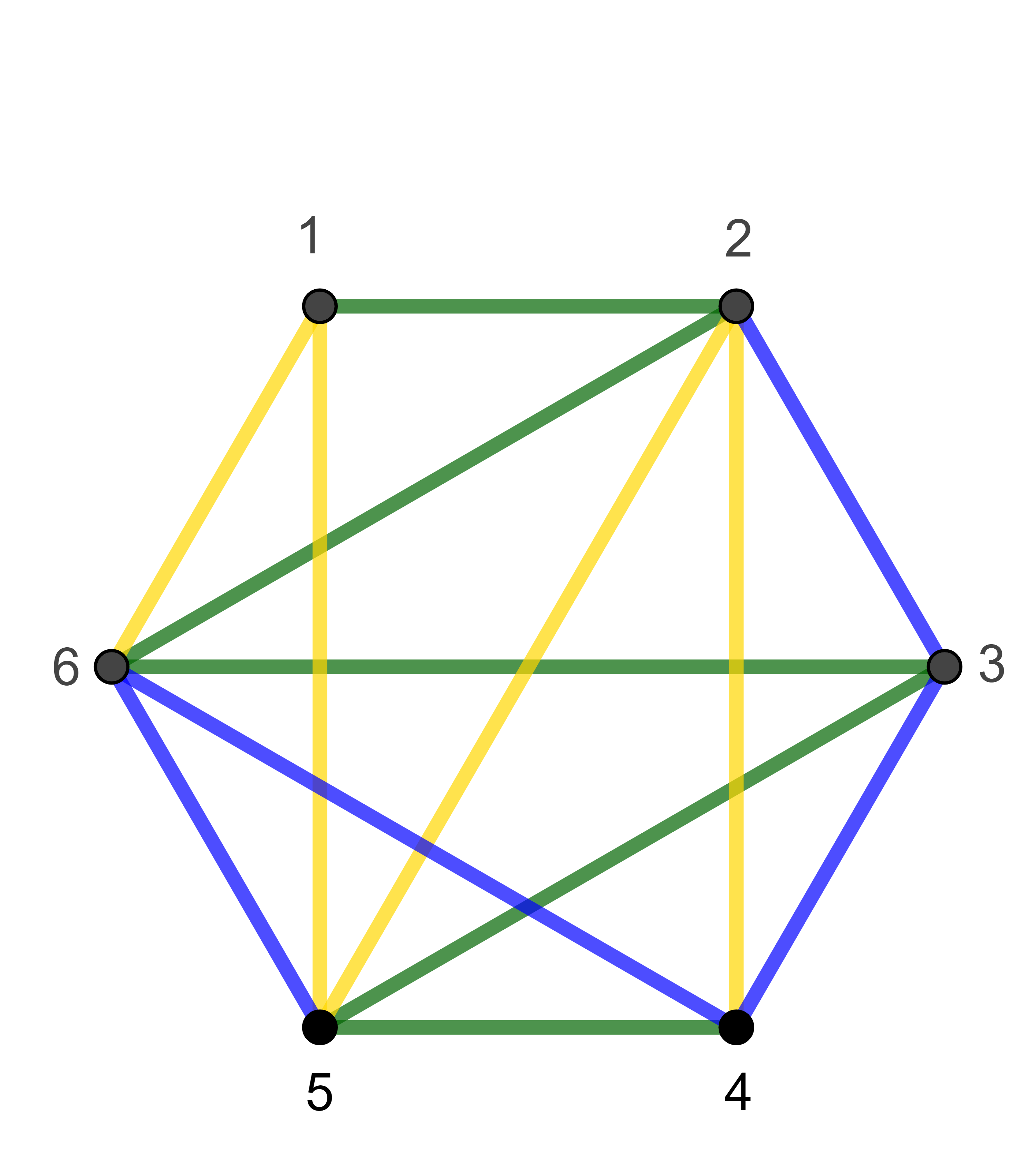}
         \caption{Reconnecting path by changing the color of base}
     \end{subfigure}
    
    \caption{Removing two consecutive edges from a complete graph with even vertexes ($K_6$)}
    \label{fig:edgeremoval}
\end{figure}

Additionally, we can remove the lone end edge that connected to $v$ without disconnecting any paths. Thus, we can remove a star with up to $2k-2$ edges (centered at $v$) from $K_{2k}$ and still have the same number of paths in the path decomposition, verifying Gallai's Conjecture.

\end{itemize}

\end{proof}

\subsection{Removing tadpoles}
We now turn to tadpoles as another structure that can be removed from complete graphs, with Gallai's Conjecture still holding. The $T_{m,n}$ tadpole graph~\cite{vaidya11}, is the graph obtained by joining any vertex of $C_m$ to one end of $P_n$. Figure \ref{fig:T51} shows $T_{5,1}$.

\begin{figure}[htp]
    \centering
    \includegraphics[height=3cm]{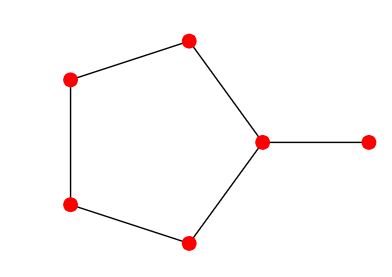}
    \caption{Tadpole, $T_{5,1}$}
    \label{fig:T51}
\end{figure}
\begin{theorem}
Given $m \leq n-1$ for $n \geq 2$, Gallai's Conjecture holds for any graph $G$ obtained from removing a tadpole $T_{m,1}$ from $K_n$.
\end{theorem}
\newpage
\begin{proof}

As before, we consider two cases:

\begin{itemize}

\item $n=2k$:

Consider our original path decomposition for $K_{2k}$. We will now provide a construction for a $T_{m,1}$ that can be removed so that the remaining path decomposition still satisfies Gallai's Conjecture. The first portion of $T_{m,1}$ is the path $v_1v_2 \dots v_{m}$ of length $m-1$. 

Now, let $P_1$ be the path in the construction for $K_{2k}$ that contains $v_{m}v_{2k}$. We now form a fork with $v_{m}v_{2k}$ and the edge on $P_1$ that does not contain both vertices in the cycle. Then, removing the fork and cycle and reassigning the base of the fork to $P_2$ completes the removal of $T_{m,1}$. Note that the base will be an end edge and thus no paths are disconnected when the path of the base is changed to $P_1$. Our new graph and decomposition still satisfy Gallai's Conjecture. 

Figure~\ref{fig:tadpoleremoval} shows the process of removing a $T_{6, 1}$ from $K_8$. 
\begin{figure}[htp]
    \centering
     \begin{subfigure}[t]{0.32\textwidth}
         \centering
         \includegraphics[width=\textwidth]{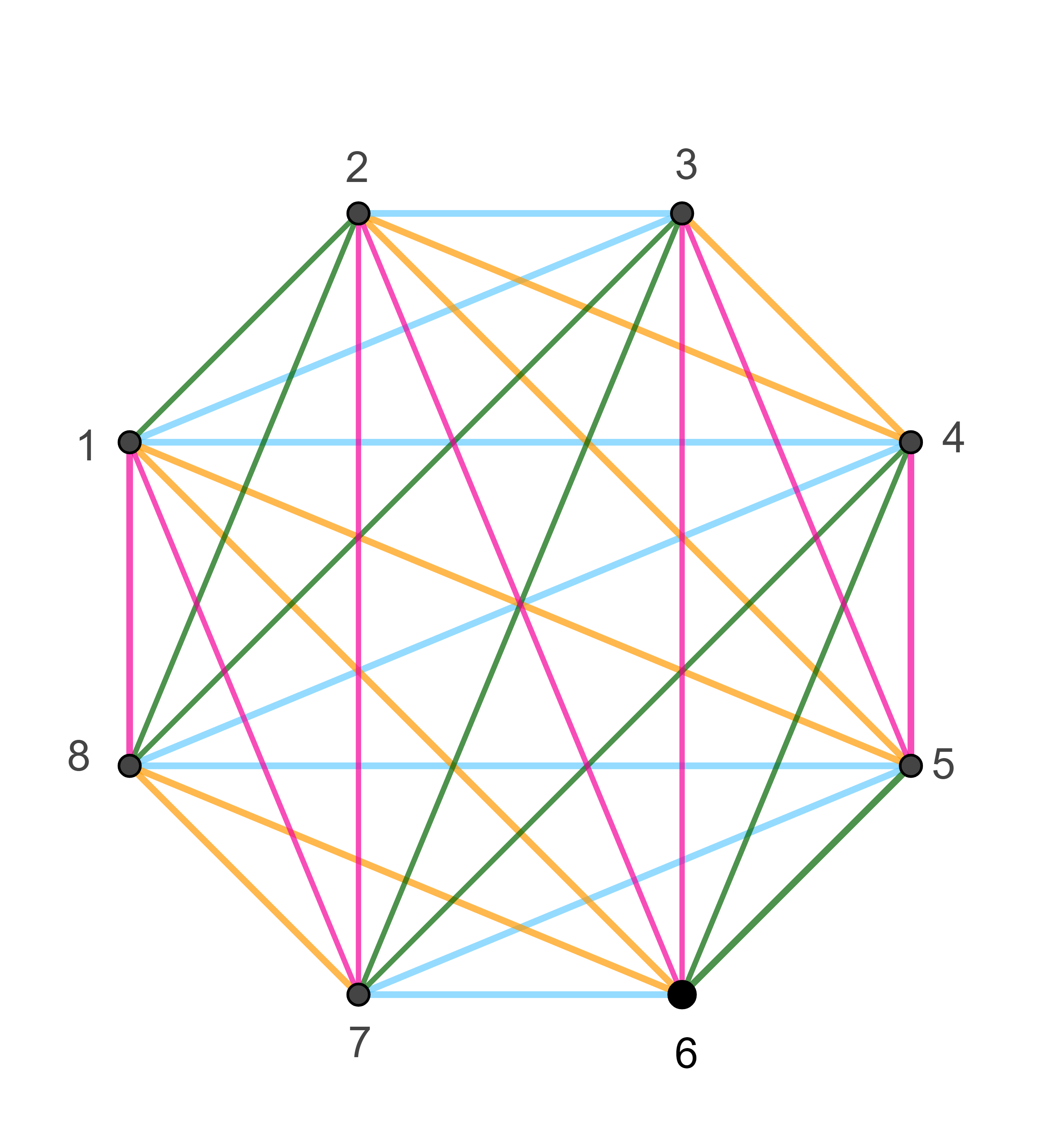}         \caption{$K_8$}
     \end{subfigure}
     \hfill
     \begin{subfigure}[t]{0.32\textwidth}
         \centering
         \includegraphics[width=\textwidth]{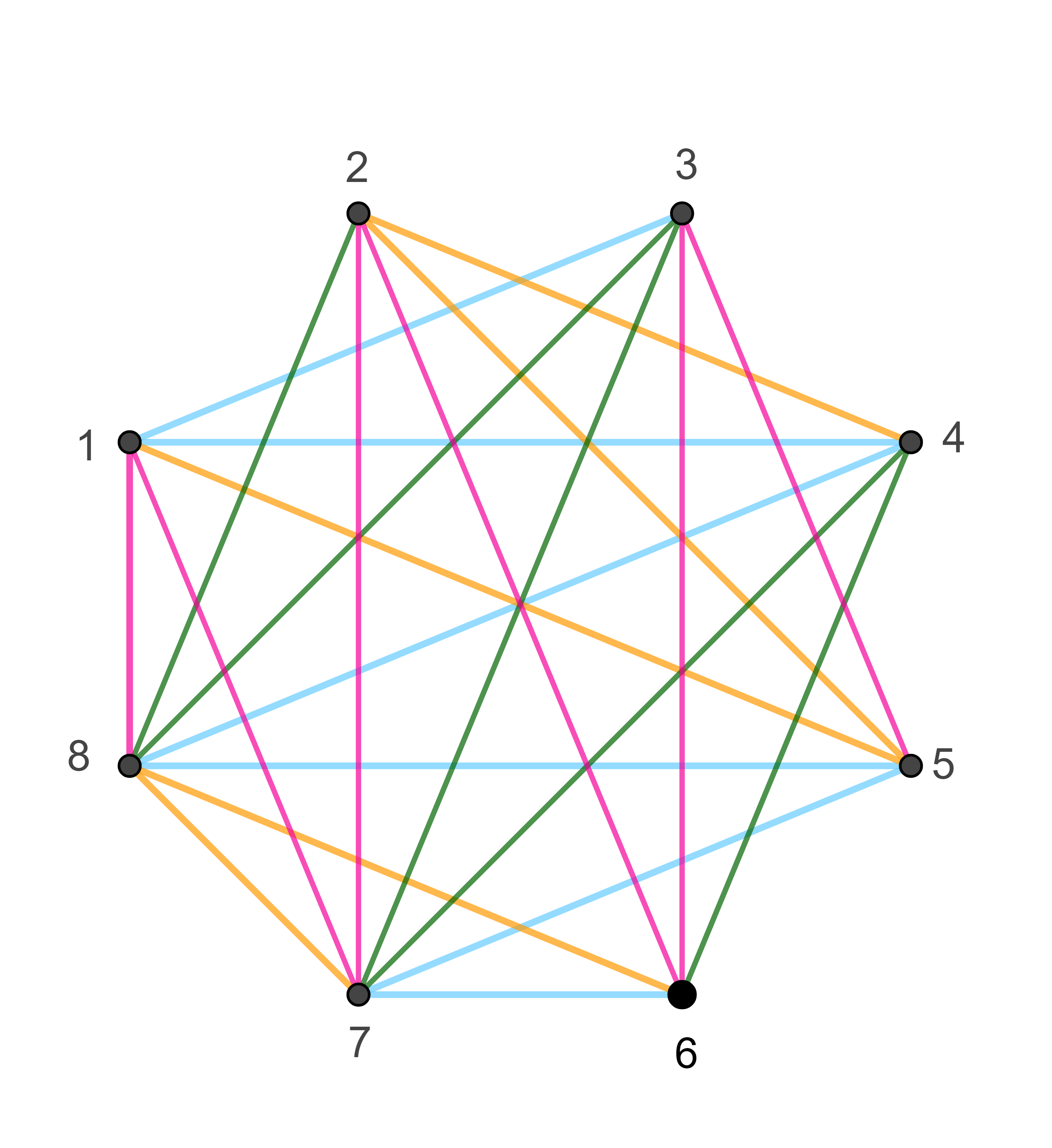}
         \caption{Removing $v_1v_2v_3v_4v_5v_6v_1$}
     \end{subfigure}
     \hfill
     \begin{subfigure}[t]{0.32\textwidth}
         \centering
         \includegraphics[width=\textwidth]{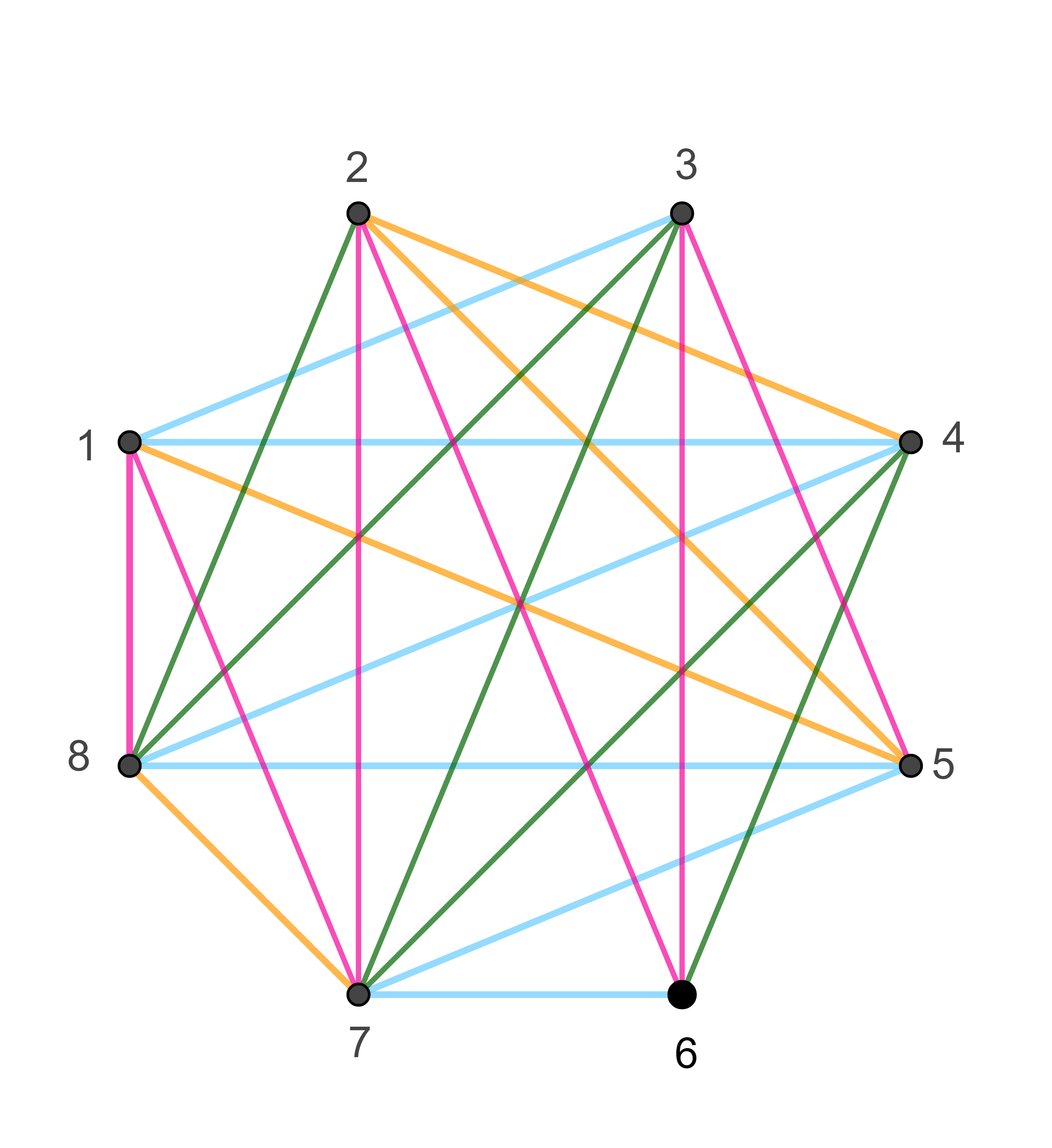}
         \caption{Removing $v_6v_8$}
     \end{subfigure}

         \begin{subfigure}[t]{0.32\textwidth}
         \centering
         \includegraphics[width=\textwidth]{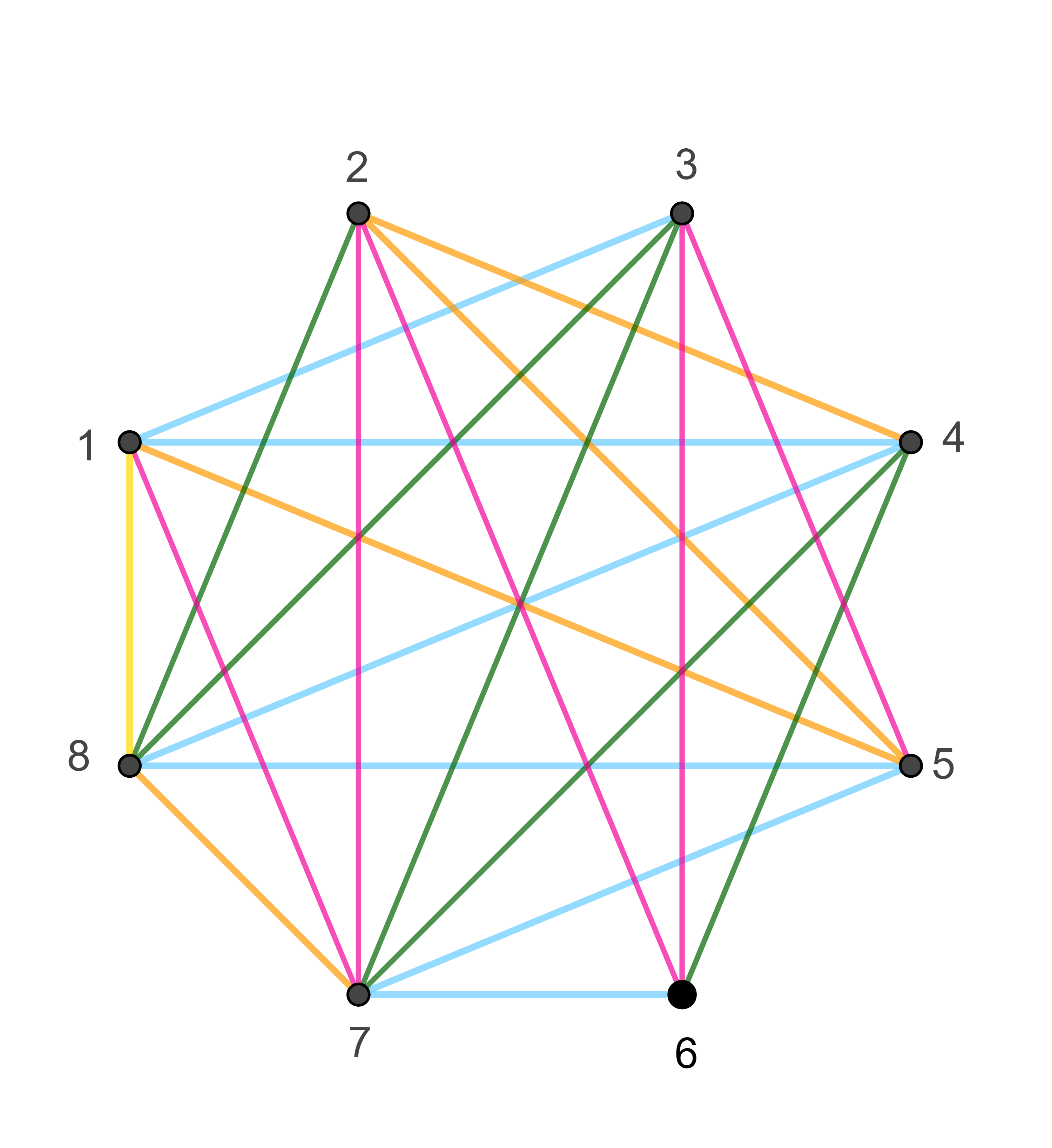}
         \caption{Switching $v_8v_1$ to orange path}
     \end{subfigure}
     \hskip 0.3in
     \begin{subfigure}[t]{0.32\textwidth}
         \centering
         \includegraphics[width=\textwidth]{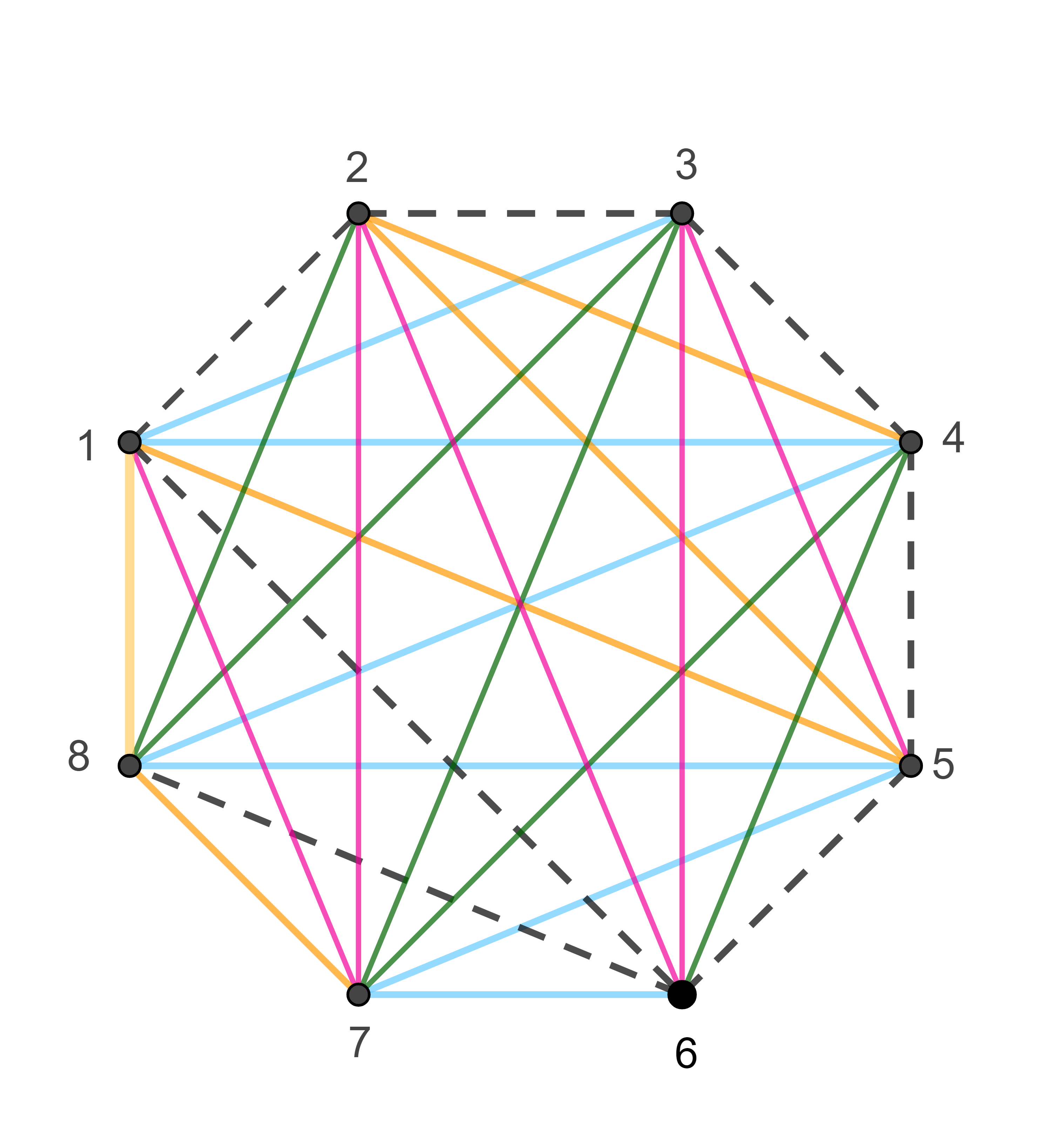}
         \caption{Removed tadpole shown in dotted line}
     \end{subfigure}
    \caption{Removing a tadpole, $T_{6,1}$, from $K_8$}
    \label{fig:tadpoleremoval}
\end{figure}

\item $n=2k+1$:

For odd $k$, again, we start with our initial construction for the path decomposition of such a complete graph. For a cycle of length $m$ where $m < 2k$ we use the same construction as with the even case. For $m = 2k$ we use the cycle $v_1v_2\dots v_{2k}$ and let the hanging edge be $v_{2k}v_{2k+1}$. These constructions show that all tadpoles with cycle length up to $2k$ and path length $1$ can be removed from the complete graph and still form a graph satisfying Gallai's Conjecture. 

The argument for a complete graph with $2k+1$ vertices where $k$ is even is similar. 

Figure~\ref{fig:tadpoleodd} shows the process of removing a $T_{4, 1}$ from $K_7$. 

\begin{figure}[htp]
    \centering
     \begin{subfigure}[t]{0.3\textwidth}
         \centering
         \includegraphics[width=\textwidth]{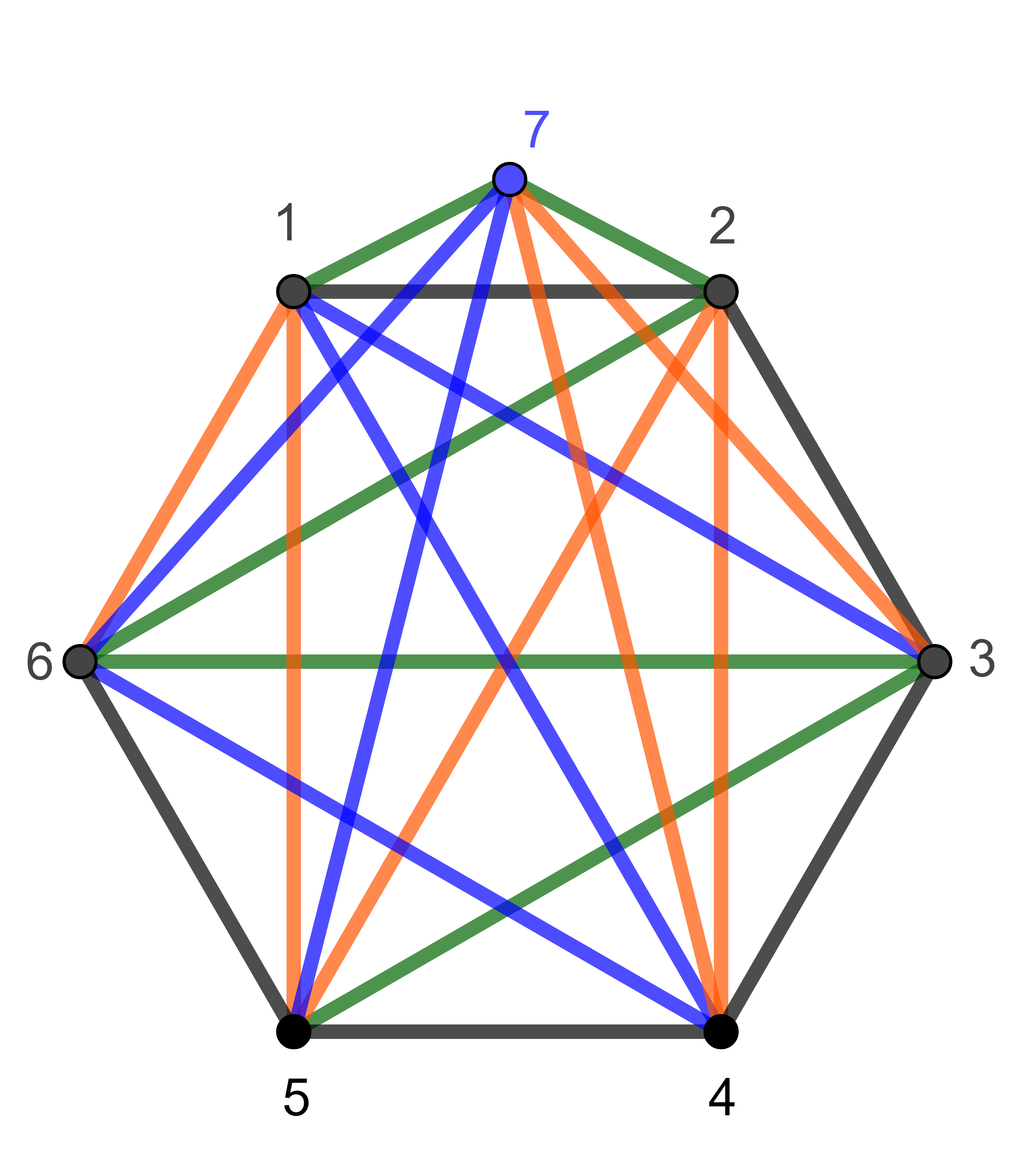}       
         \caption{$K_7$}
     \end{subfigure}
     \hfill
     \begin{subfigure}[t]{0.3\textwidth}
         \centering
         \includegraphics[width=\textwidth]{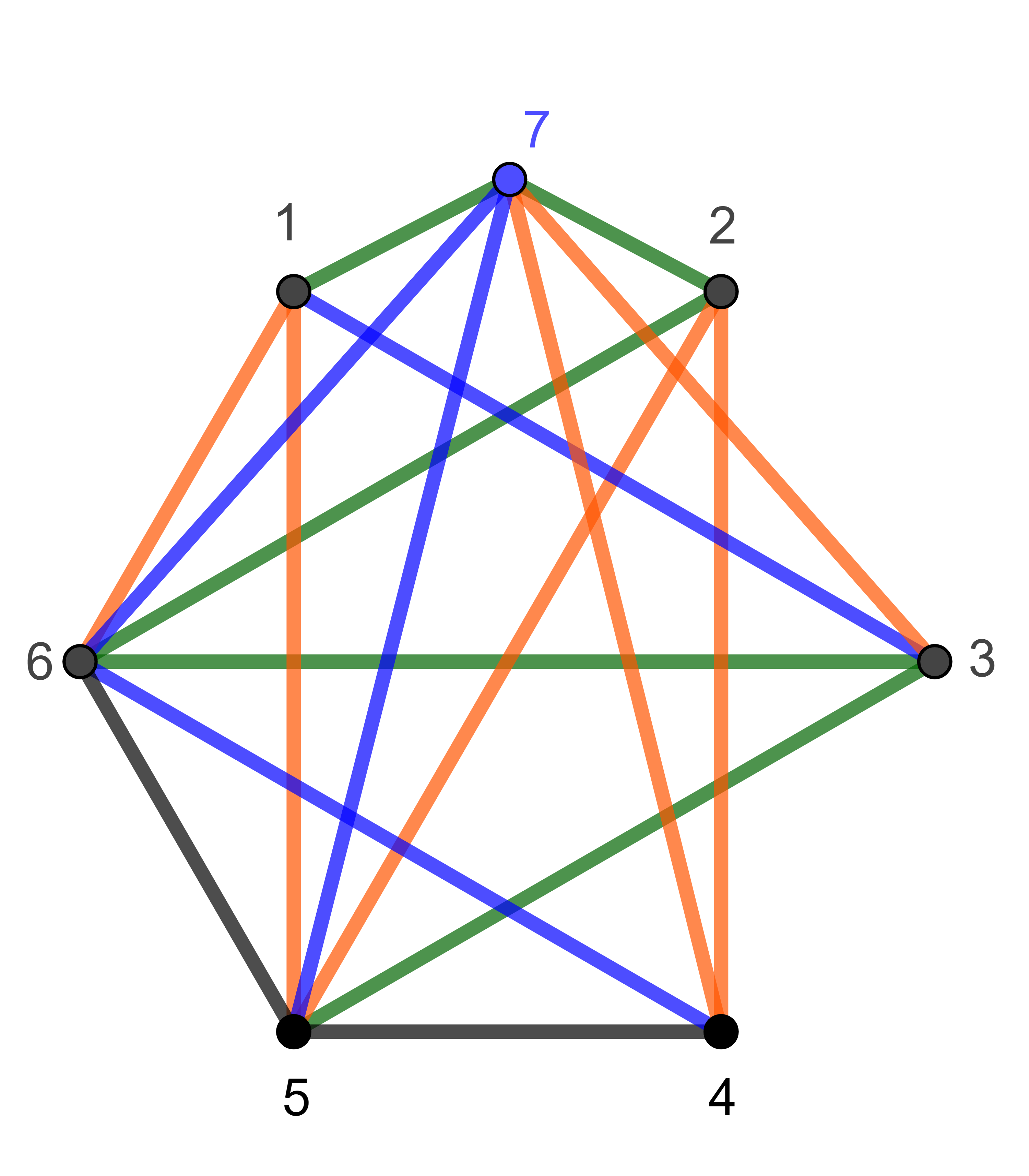}
         \caption{Removing $v_1v_2v_3v_4v_1$}
     \end{subfigure}
     \hfill
     \begin{subfigure}[t]{0.3\textwidth}
         \centering
         \includegraphics[width=\textwidth]{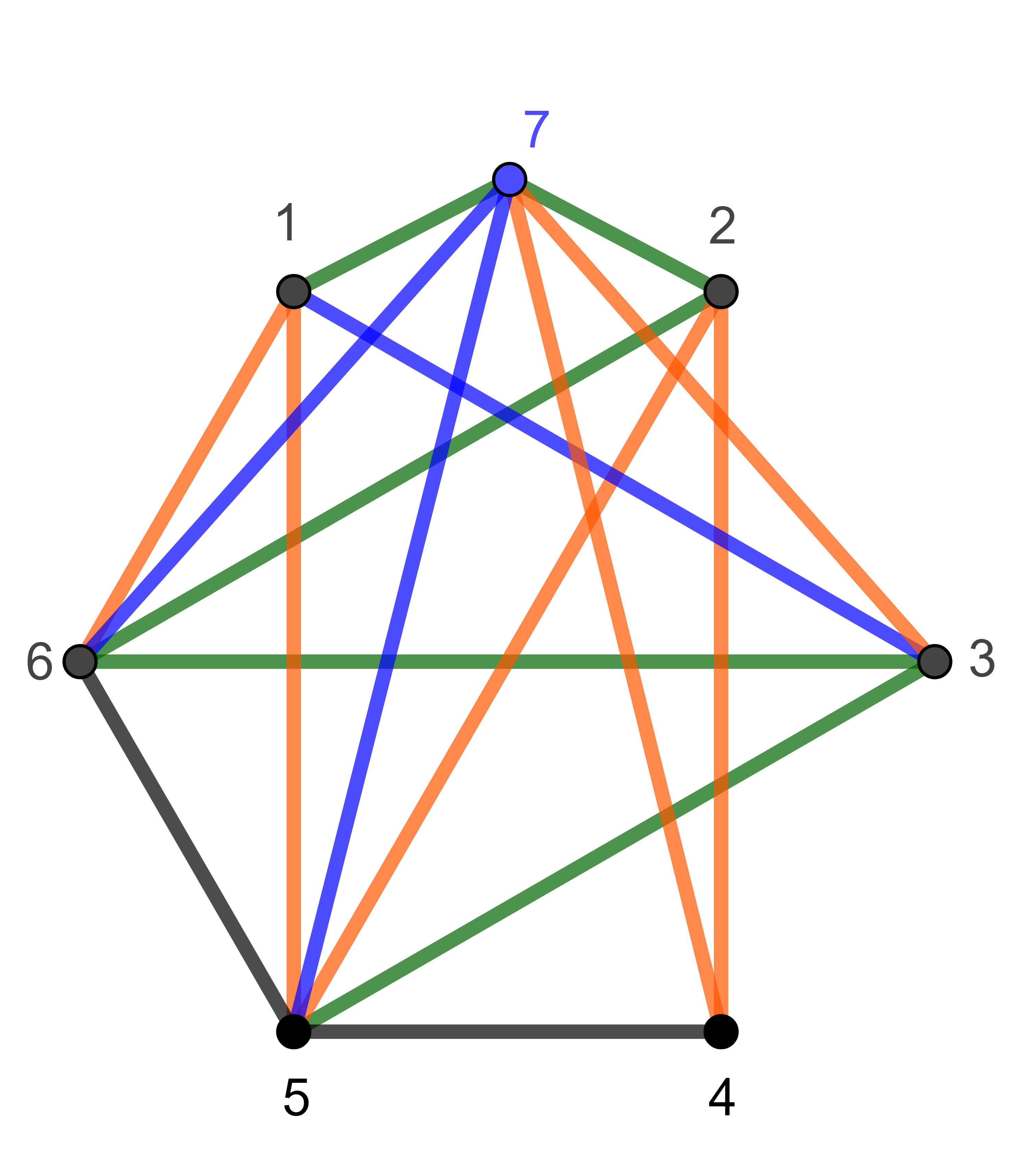}
         \caption{Removing $v_4v_6$}
     \end{subfigure}

         \begin{subfigure}[t]{0.3\textwidth}
         \centering
         \includegraphics[width=\textwidth]{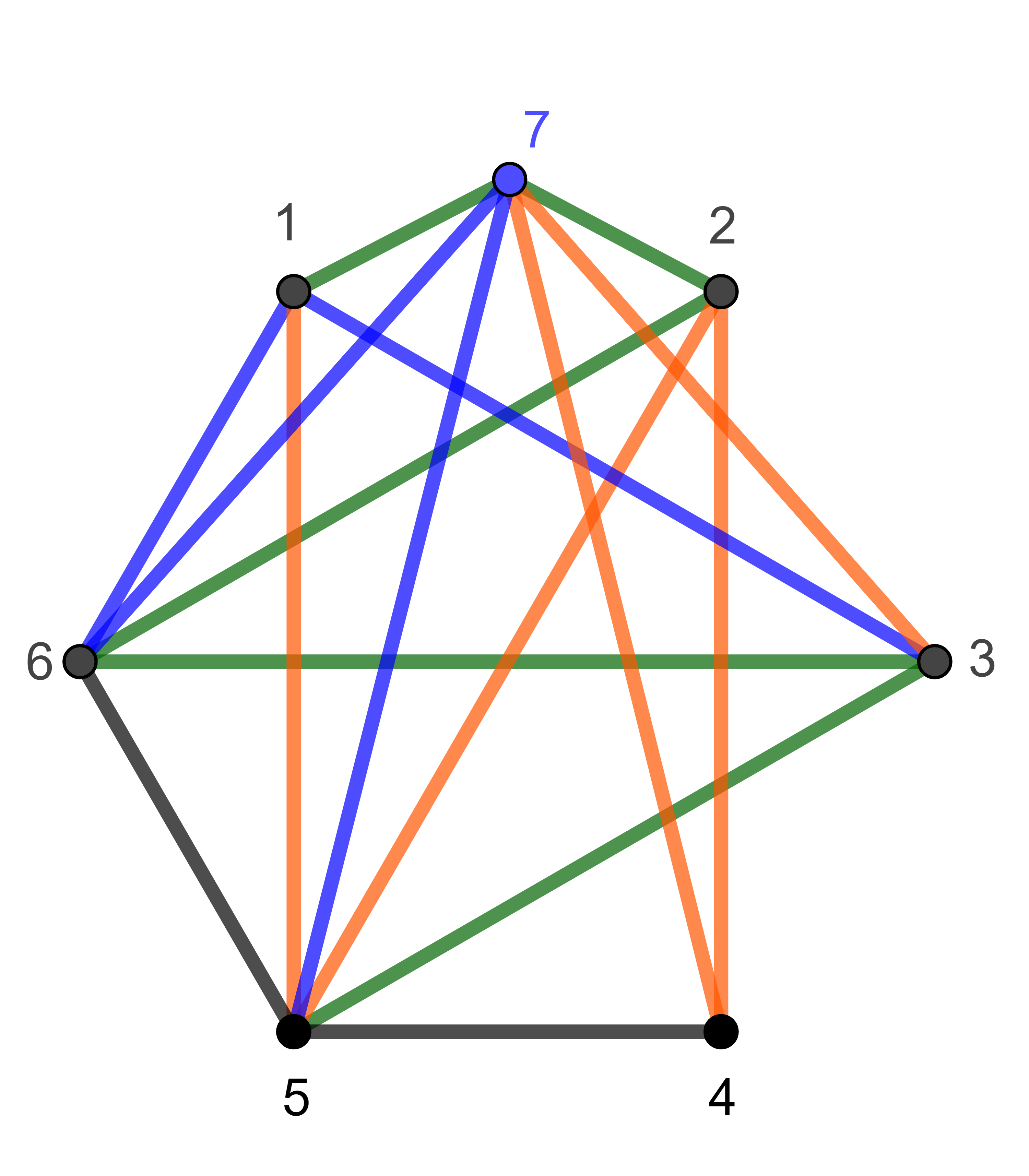}
         \caption{Changing $v_1v_6$ to blue path}
     \end{subfigure}
     \hskip 0.3in
     \begin{subfigure}[t]{0.3\textwidth}
         \centering
         \includegraphics[width=\textwidth]{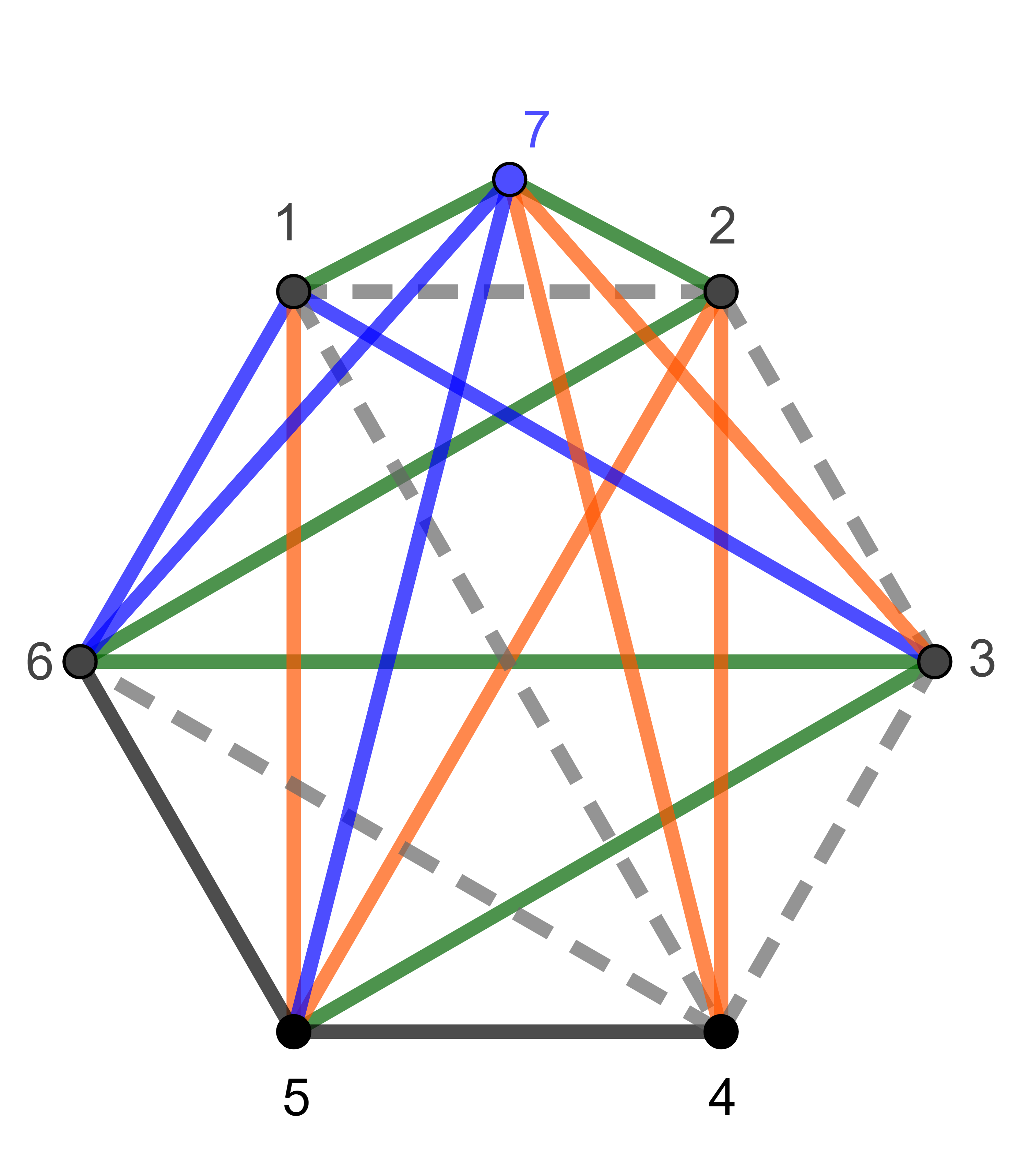}
         \caption{Dotted line shows removed tadpole}
     \end{subfigure}
    \caption{Removing a tadpole, $T_{4,1}$, from $K_7$}
    \label{fig:tadpoleodd}
\end{figure}

\end{itemize}

\end{proof}

\newpage
\section{Concluding remarks}\label{sec:con}
In this paper, we presented explicit constructions for a path decomposition that satisfied Gallai's Conjecture for all complete graphs. Using these constructions, we show that stars and certain tadpoles could be removed from any complete graph so that the remaining decomposition still satisfies Gallai's Conjecture. It seems that determining types of graphs that can be removed from complete graphs while still preserving the property that they satisfy Gallai's Conjecture can be a viable method in general. These families of graphs are not covered by previous research results that are bounded by degree restrictions, highlighting the usefulness of our approach to Gallai's Conjecture. 


In general, given a path decomposition for a graph, we can remove any number of edges from the ends of any paths with the remaining decomposition still satisfying Gallai's Conjecture. 
Consequently, any subgraphs formed by ``path ends'' in our decomposition may be removed, leaving a new graph and decomposition with Gallai's Conjecture satisfied. This was the key idea behind the proofs in the previous section.

On the other hand, this ``path ends removal'' approach fails if a subgraph (to be removed) cannot be formed by path ends in our (or any) decomposition. To further explore this direction, it makes sense to consider all possible (non-isomorphic) path decompositions (satisfying Gallai's Conjecture) of complete graphs. 


As examples, we present three non-isomorphic path decompositions that satisfy Gallai's Conjecture for $K_6$ in Figure~\ref{fig:nonisomorphic} and two for $K_8$ in Figure~\ref{fig:nonisomorphic8} respectively.

\begin{figure}[htp]
    \centering
    \includegraphics[width=5cm]{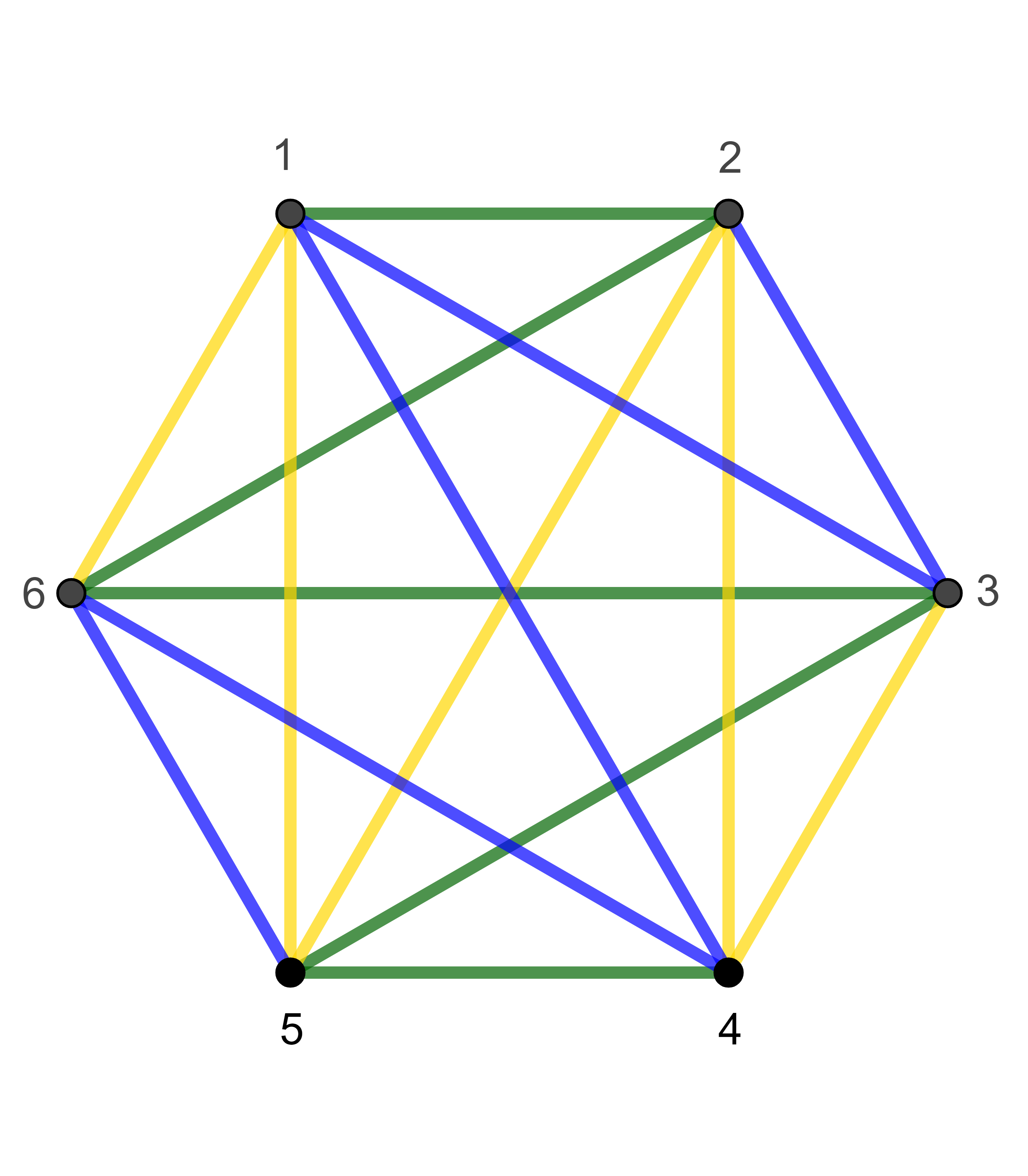}
    \includegraphics[width=5cm]{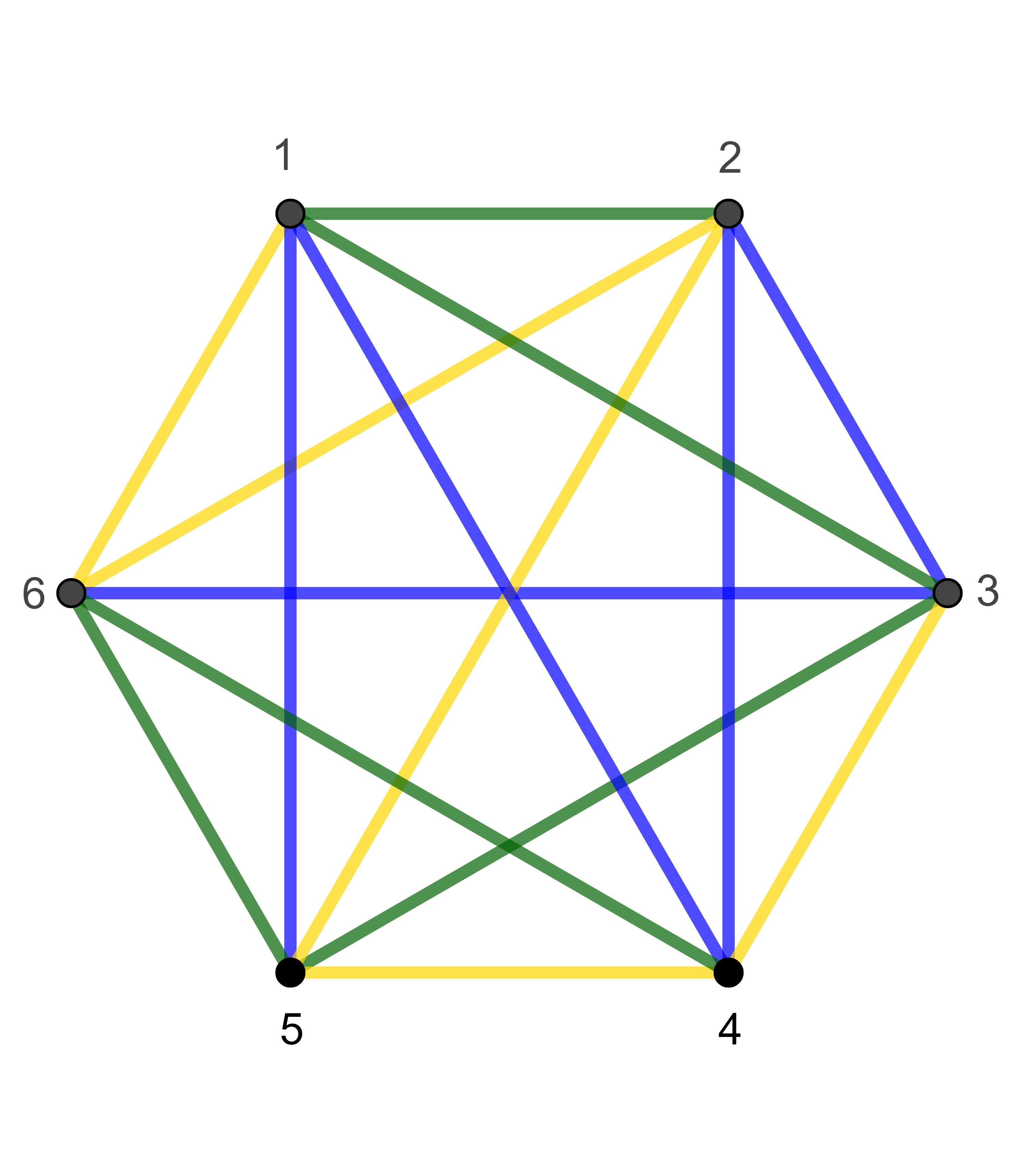}
    \includegraphics[width=5cm]{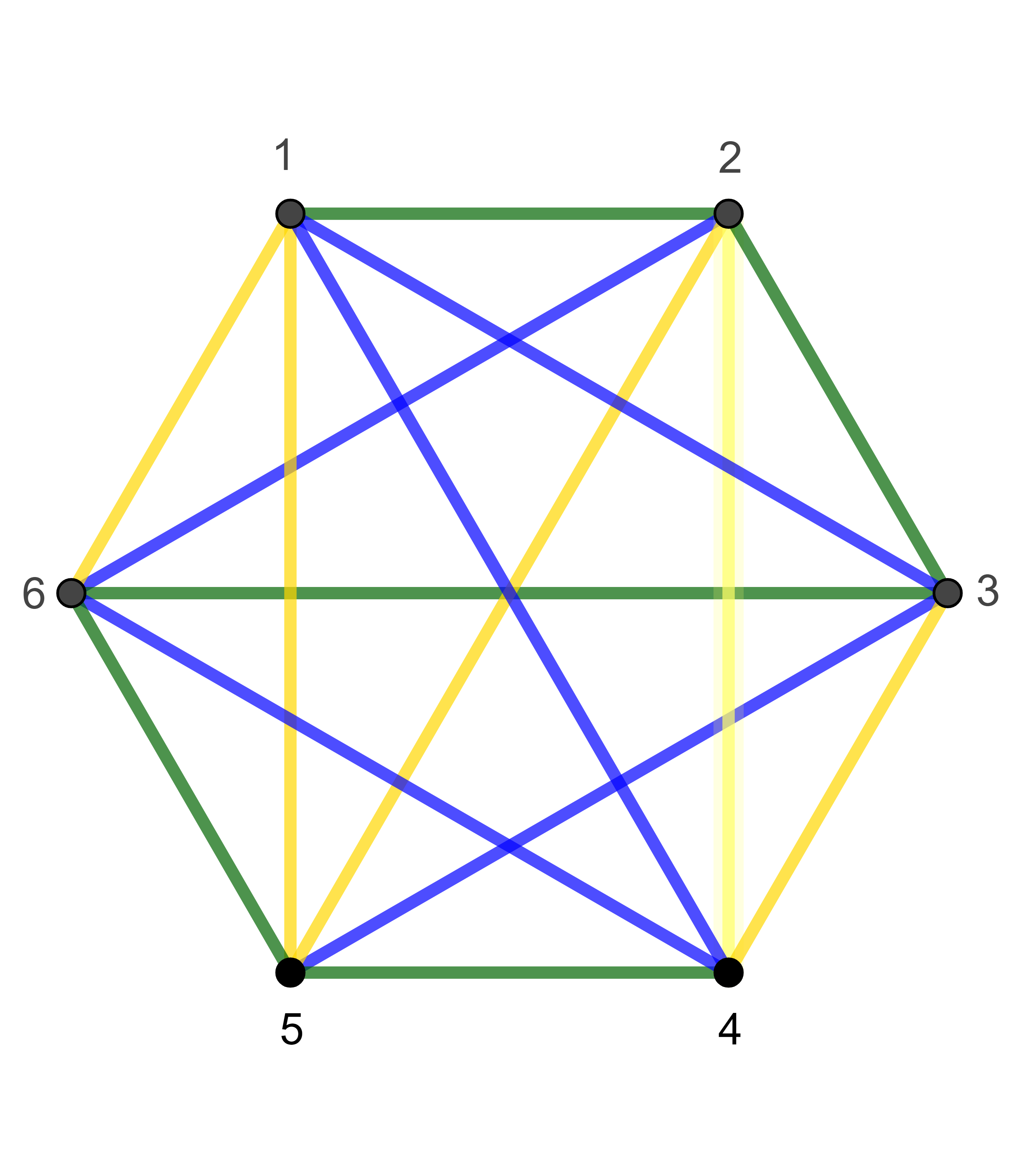}
    \caption{The three non-isomorphic path decompositions for $K_6$}
    \label{fig:nonisomorphic}
\end{figure}

\begin{figure}[htp]
    \centering
         \includegraphics[width=6.5cm]{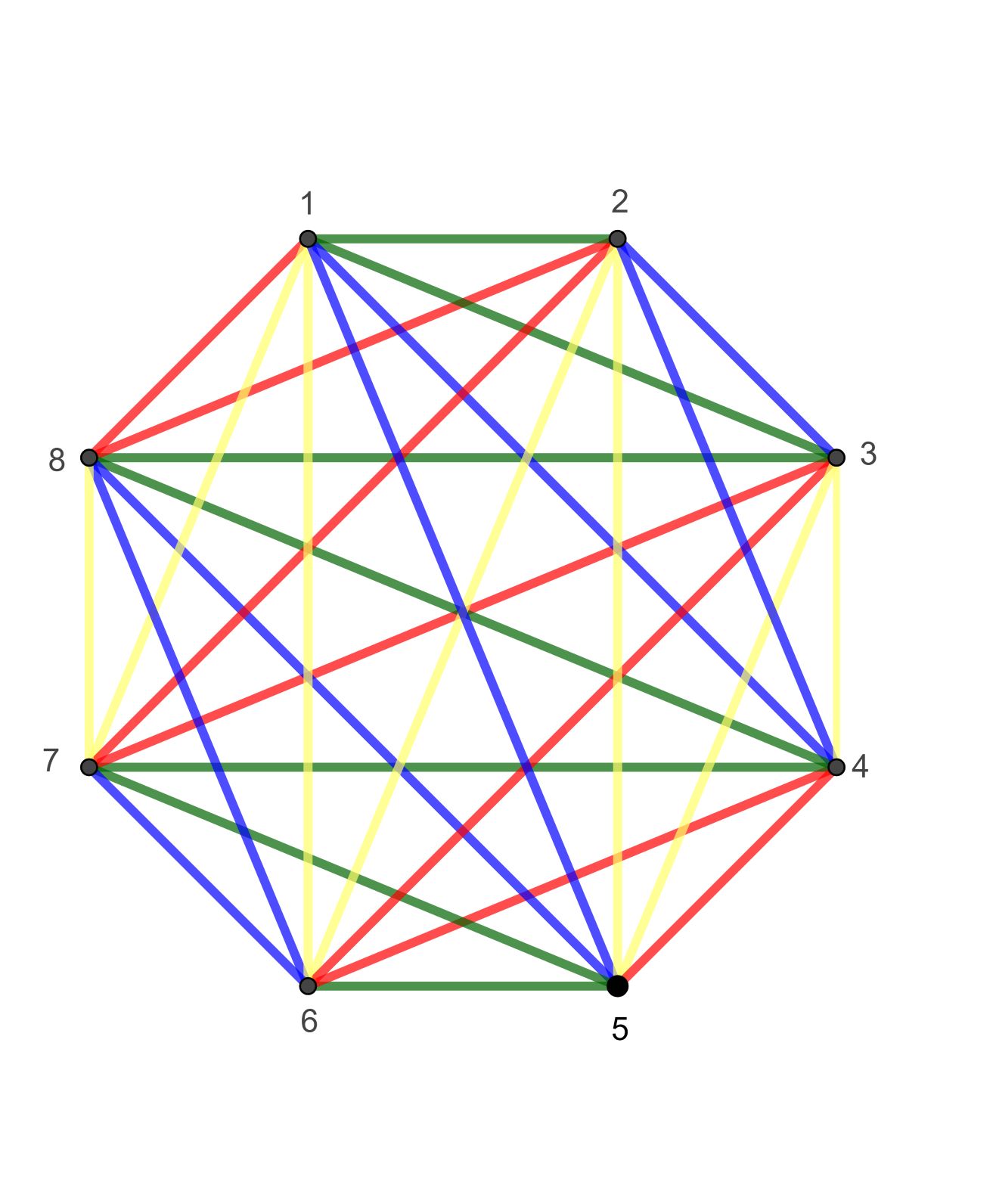}
         \includegraphics[width=6.5cm]{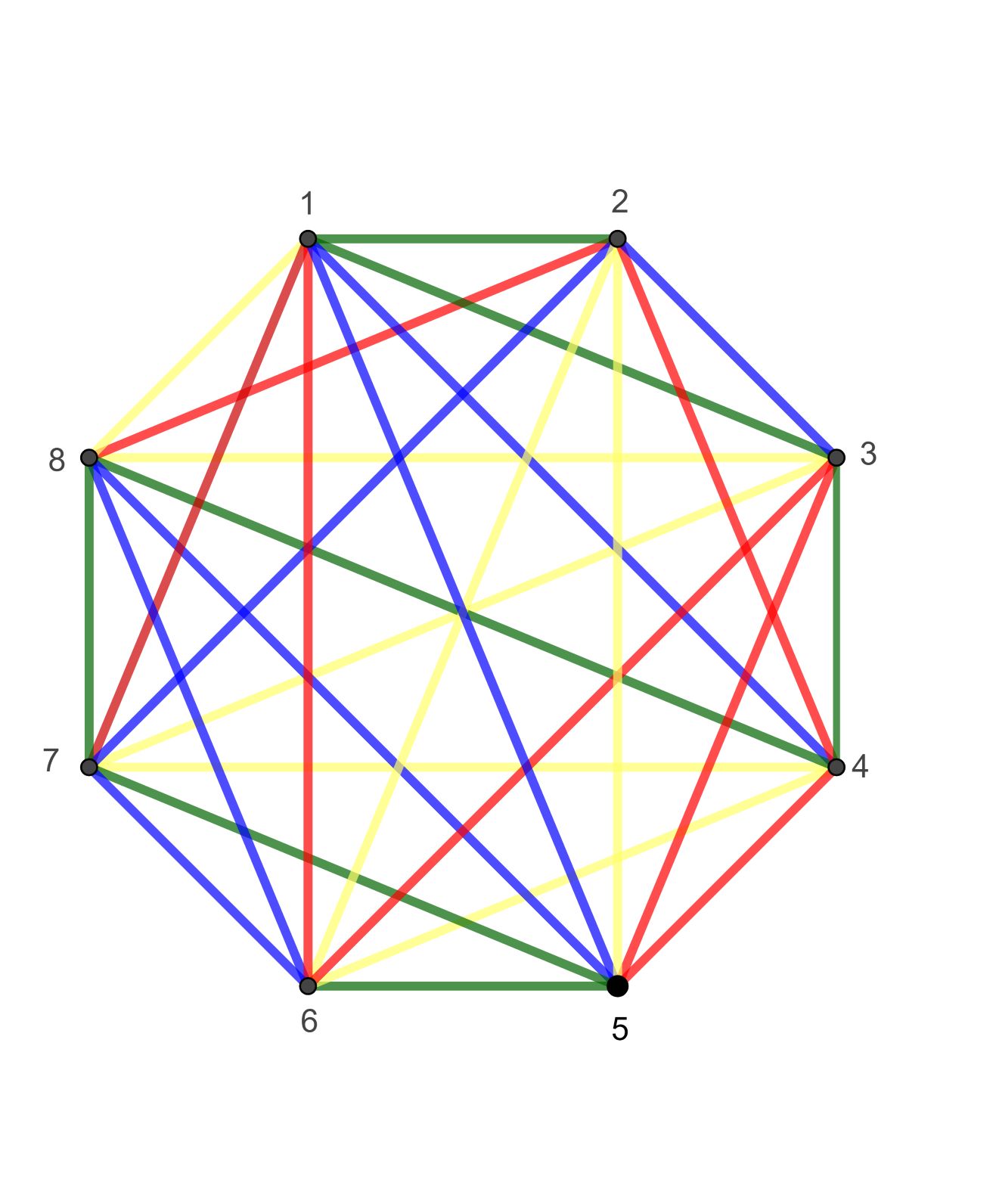}
    \caption{Two non-isomorphic path decompositions for $K_8$}
    \label{fig:nonisomorphic8}
\end{figure}

Through simple counting arguments, we can verify that the three path decompositions shown in Figure~\ref{fig:nonisomorphic} are indeed all the non-isomorphic decompositions satisfying Gallai's Conjecture. It seems to be an interesting and challenging problem to find all such path decompositions for $K_n$ for general $n$.


\section*{Acknowledgement}
The authors would like to thank Wayzata High School Honors Mentor Connection Program and DSM Academy for the opportunity of this research collaboration.


\begin{thebibliography}{30}

\bibitem{El15}
S.~El-Zanati, M.~Ermete, J.~Hasty, M.~Plantholt, and S.~Tipnis, ``On
  decomposing regular graphs into isomorphic double-stars,'' {\em Discussiones
  Mathematicae Graph Theory}, vol.~35, pp.~73 -- 79, 02 2015.

\bibitem{Gal59}
P.~Erdős and T.~Gallai, ``On maximal paths and circuits of graphs,'' {\em Acta
  Mathematica Hungarica - ACTA MATH HUNG}, vol.~10, pp.~337--356, 09 1959.

\bibitem{Aru13}
S.~Arumugam, I.~Sahul~Hamid, and V.~M. Abraham, ``Decomposition of graphs into
  paths and cycles,'' {\em Journal of Discrete Mathematics}, vol.~2013, 01
  2013.

\bibitem{lov68}
L.~László, ``On covering of graphs,'' {\em Theory of Graphs}, p.~231–236,
  01 1968.

\bibitem{lp96}
L.~Pyber, ``Covering the edges of a connected graph by paths,'' {\em J. Comb.
  Theory, Ser. B}, vol.~66, pp.~152--159, 01 1996.

\bibitem{GFan05}
G.~Fan, ``Path decompositions and gallai's conjecture,'' {\em J. Comb. Theory,
  Ser. B}, vol.~93, pp.~117--125, 03 2005.

\bibitem{Botler20}
F.~Botler and M.~Sambinelli, ``Towards gallai's path decomposition
  conjecture,'' {\em Journal of Graph Theory}, vol.~97, 11 2020.

\bibitem{bon19}
M.~Bonamy and T.~Perrett, ``Gallai's path decomposition conjecture for graphs
  of small maximum degree,'' {\em Discrete Mathematics}, vol.~342, 09 2016.

\bibitem{chu21}
Y.~Chu, G.~Fan, and Q.~Liu, ``On gallai’s conjecture for graphs with maximum
  degree 6,'' {\em Discrete Mathematics}, vol.~344, pp.~112--212, 02 2021.

\bibitem{alspach08}
B.~Alspach, ``The wonderful walecki construction,'' {\em Bulletin of the
  Institute of Combinatorics and its Applications}, vol.~52, pp.~7--20, 01
  2008.

\bibitem{vaidya11}
S.~Vaidya and K.~Kanani, ``Some strongly multiplicative graphs in the context
  of arbitrary super subdivision,'' {\em International Journal of Applied
  Mathematics and Computation ,2011 ISSN 0974 - 4665 (print version) ISSN 0974
  - 4673 (electronic version)}, vol.~3, pp.~60--64, 01 2011.
\end{thebibliography}

\end{document}